\documentclass[10pt,a4paper]{article}
\usepackage[latin1]{inputenc}
\usepackage{anysize}
\usepackage{graphicx}

\usepackage{amssymb,amsmath,amscd,amsfonts,amsthm,mathrsfs,color}

\newcommand{\cc}[1]{\mathcal{#1}}

\newcommand{\ra}[1]{\overrightarrow{#1}}
\newcommand{\la}[1]{\overleftarrow{#1}}

\MakeRobust{\overrightarrow}

\numberwithin{equation}{section}
\newtheorem{theorem}{Theorem}
\newtheorem{lemma}[theorem]{Lemma}
\newtheorem{corollary}[theorem]{Corollary}
\newtheorem{conjecture}[theorem]{Conjecture}
\newtheorem{question}[theorem]{Question}
\usepackage{url}
\begin{document}
\title{Refining a Tree-Decomposition which Distinguishes Tangles}
\author{Joshua Erde}
\maketitle
\begin{abstract}
Roberston and Seymour introduced tangles of order $k$ as objects representing highly connected parts of a graph and showed that every graph admits a tree-decomposition of adhesion $<\!k$ in which each tangle of order $k$ is contained in a different part. Recently, Carmesin, Diestel, Hamann and Hundertmark showed that such a tree-decomposition can be constructed in a canonical way, which makes it invariant under automorphisms of the graph. These canonical tree-decompositions necessarily have parts which contain no tangle of order $k$, which we call inessential. Diestel asked what could be said about the structure of the inessential parts. In this paper we show that the torsos of the inessential parts in these tree-decompositions have branch-width $<\!k$, allowing us to further refine the canonical tree-decompositions, and also show that a similar result holds for $k$-blocks. We also use our methods to further refine the essential parts in such a tree-decomposition in a similar fashion.
\end{abstract}
\section{Introduction}
A classical notion in graph theory is that of the block-cut vertex tree of a graph. It tells us that if we consider the maximal $2$-connected components of a connected graph $G$ then they are arranged in a `tree-like' manner, separated by the cut vertices of $G$. A result of Tutte's \cite{TutteGrTh} says that we can decompose any $2$-connected graph in a similar way. Broadly, it says that every $2$-connected graph can be decomposed in a `tree-like' manner, so that the parts are separated by vertex sets of size at most $2$, and every part, together with the edges in the separators adjacent to it, is either $3$-connected or a cycle. We call the union of a part and the edges in the separators adjacent to it the \emph{torso} of the part. In contrast to the first example not every part, or even torso, of this decomposition is $3$-connected, and indeed it is easy to show that not every $2$-connected graph can be decomposed in this way such that every torso is $3$-connected. 

It has long been an open problem how best to extend these results for general $k$, the aim being to decompose a $(k-1)$-connected graph into its `$k$-connected components', where the precise meaning of what these `$k$-connected components' should be considered to be has varied. Tutte's example shows us that there may be parts of this decomposition which are not highly connected, but rather play a structural role in the graph of linking the highly connected parts together, and further that the highly connected parts of the decomposition may not correspond exactly to $k$-connected subgraphs.

Whereas initially these `$k$-connected components' were considered as concrete structures in the graph itself, Robertson and Seymour \cite{GMX} radically re-interpreted them as \emph{tangles of order $k$}, which for brevity we will refer to as \emph{$k$-tangles}\footnote{Precise definitions of many of the terms in the introduction will be postponed until Section \ref{s:back}, where all the necessary background material will be introduced.}. Instead of being defined in terms of the edges and vertices of a graph, these objects were defined in terms of structures on the set of low-order separations of a graph. 

Robertson and Seymour showed that, given any set of distinct $k$-tangles $T_1,T_2, \ldots , T_n$ in a graph $G$, there is a tree-decomposition $(T,\cc{V})$ of $G$ with precisely $n$ parts in which the orientations induced by the tangles $T_i$ on $E(T)$ each have distinct sinks, where we say the tangle is \emph{contained} in this sink. We say that such a tree-decomposition \emph{distinguishes} the tangles $T_1,T_2, \ldots, T_n$. They showed further that these tree-decompositions can be chosen so that the separators between the parts are in some way minimal with respect to the tangles considered. We say that such a tree-decomposition distinguishes the $k$-tangles \emph{efficiently}. If we call the largest size of a separator in a tree-decomposition the \emph{adhesion} of the tree-decomposition, then in particular their result implies the following:
\begin{theorem}[Robertson and Seymour \cite{GMX}]\label{t:RStan}
For every graph $G$ and $k \geq 2$ there exists a tree-decomposition $(T,\cc{V})$ of $G$ of adhesion $<\!k$ which distinguishes the set of $k$-tangles in $G$ efficiently.
\end{theorem}
More recently Carmesin, Diestel, Hamann and Hundertmark \cite{CDHH13CanonicalAlg} described a family of algorithms that can be used to build tree-decompositions which distinguish the set of $k$-tangles in a graph and are \emph{canonical}, that is, they are invariant under every automorphism of the graph. 

Just as in Tutte's theorem, where there were parts of the tree-decomposition whose torsos were not 3-connected, it is easy to show that the tree-decompositions formed in \cite{CDHH13CanonicalAlg} must contain parts which do not contain any $k$-tangle. Since the general motivation for these tree-decompositions is to decompose the graph into its `$k$-connected components' in a way that displays the global structure of the graph, it is natural to ask further questions about the structure of these tree-decompositions. In \cite{CDHH13CanonicalParts} Carmesin et al. analysed the structure of the trees that the various algorithms given in \cite{CDHH13CanonicalAlg} produced. One particular question that was asked is what can be said about the structure of the parts which do not contain a $k$-tangle. We will call the parts of a tree-decomposition that contain a $k$-tangle \emph{essential}, and those that do not \emph{inessential}.

For example, if the whole graph contains no $k$-tangle, then these canonical tree-decompositions tell us nothing about the graph, as they consist of just one inessential part. However there are theorems which describe the structure of a graph which contains no $k$-tangle. In the same paper where they introduced the concept of tangles, Roberston and Seymour \cite{GMX} showed that a graph which contains no $k$-tangle has branch-width $<\!k$, and in fact that the converse is also true, a graph with branch-width $\geq k$ contains a $k$-tangle. Having branch-width $< k$ can be rephrased in terms of the existence of a certain type of tree-decomposition (See e.g. \cite{DiestelOumDualityIb}). A nice property of these tree-decompositions is that each of the parts is in some sense `too small' to contain a $k$-tangle. In this way these tree-decompositions witness that a graph has no $k$-tangle by splitting the graph into a number of parts, each of which cannot contain a $k$-tangle and similarly a $k$-tangle witnesses that a graph does not have such a tree-decomposition.

A natural question to then ask is, do the inessential parts in the tree-decompositions from \cite{CDHH13CanonicalAlg} admit tree-decompositions of the same form, into parts which are too small to contain a $k$-tangle? If so we might hope to refine these canonical tree-decompositions by decomposing further the inessential parts. By combining these decompositions we would get an overall tree-decomposition of $G$ consisting of some essential parts, each containing a $k$-tangle in $G$, and some inessential parts, each of which is `small' enough to witness the fact that no $k$-tangle is contained in that part.

We first note that we cannot hope for these refinements to also be canonical. For example consider a graph formed by taking a large cycle $C$ and adjoining to each edge a large complete graph $K_n$. Then a canonical tree-decomposition which distinguishes the $3$-tangles in this graph will contain the cycle $C$ as an inessential part. However there is no canonical tree-decomposition of $C$ with branch-width $<3$. Indeed, if such a tree-decomposition contained any of the $2$-separations of $C$ as an adhesion set then, since all the rotations of $C$ lie in the automorphism group of $G$, every rotation of this separation must appear as an adhesion set. However these separations cannot all appear as the adhesion sets in any tree-decomposition, as every pair of vertices in a $2$-separation of $C$ are themselves separated by some rotation of that separation.

If we drop the restriction that the refinement be canonical then, at first glance, it might seem like there should clearly be such a refinement. If there is no $k$-tangle contained in a part $V_t$ in a tree-decomposition, $(T,\cc{V})$, then by the theorem of Robertson and Seymour there should be a tree-decomposition of that part with branch-width $<\!k$. However there is a problem with this naive approach, in that we have no guarantee that we can insert the tree-decomposition of this part into the existing tree-decomposition. In particular it could be the case that this tree-decomposition splits up the separators of the part $V_t$ in $(T,\cc{V})$. One way to avoid this problem is to instead consider the torso of the part $V_t$. If we have a tree-decomposition of the torso we can insert it into the original tree-decomposition, but it is not clear that adding these extra edges can not increase the branch-width of the part. In fact it is easy to find examples where choosing a bad canonical tree-decomposition to distinguish the set of $k$-tangles in a graph results in inessential parts whose torsos have branch-width $\geq k$.

For example consider the following graph: We start with the union of three large complete graphs, $K_{N_1}$, $K_{N_2}$ and $K_{N_3}$, for $N_1,N_2 ,N_3>>k$. We pick a set of $(k-1)/2$ vertices from each graph, which we denote by $X_1$, $X_2$ and $X_3$ respectively, and join each of these sets completely to a new vertex $x$.  It is a simple check that there are three $k$-tangles in this graph, corresponding to the three large complete subgraphs. However, consider the following tree-decomposition of the graph into four parts $K_{N_1} \cup X_2$, $K_{N_2} \cup X_3$, $K_{N_3} \cup X_1$ and $X_1 \cup X_2 \cup X_3 \cup \{x\}$. This is a tree-decomposition which distinguishes the $k$-tangles in the graph, and the part $X_1 \cup X_2 \cup X_3 \cup \{x\}$ is inessential. However the torso of this middle part is a complete graph of order $3(k-1)/2 + 1$, which can be seen to have branch-width $\geq k$. 

\begin{figure}[!ht]
\center
\includegraphics[scale=0.2]{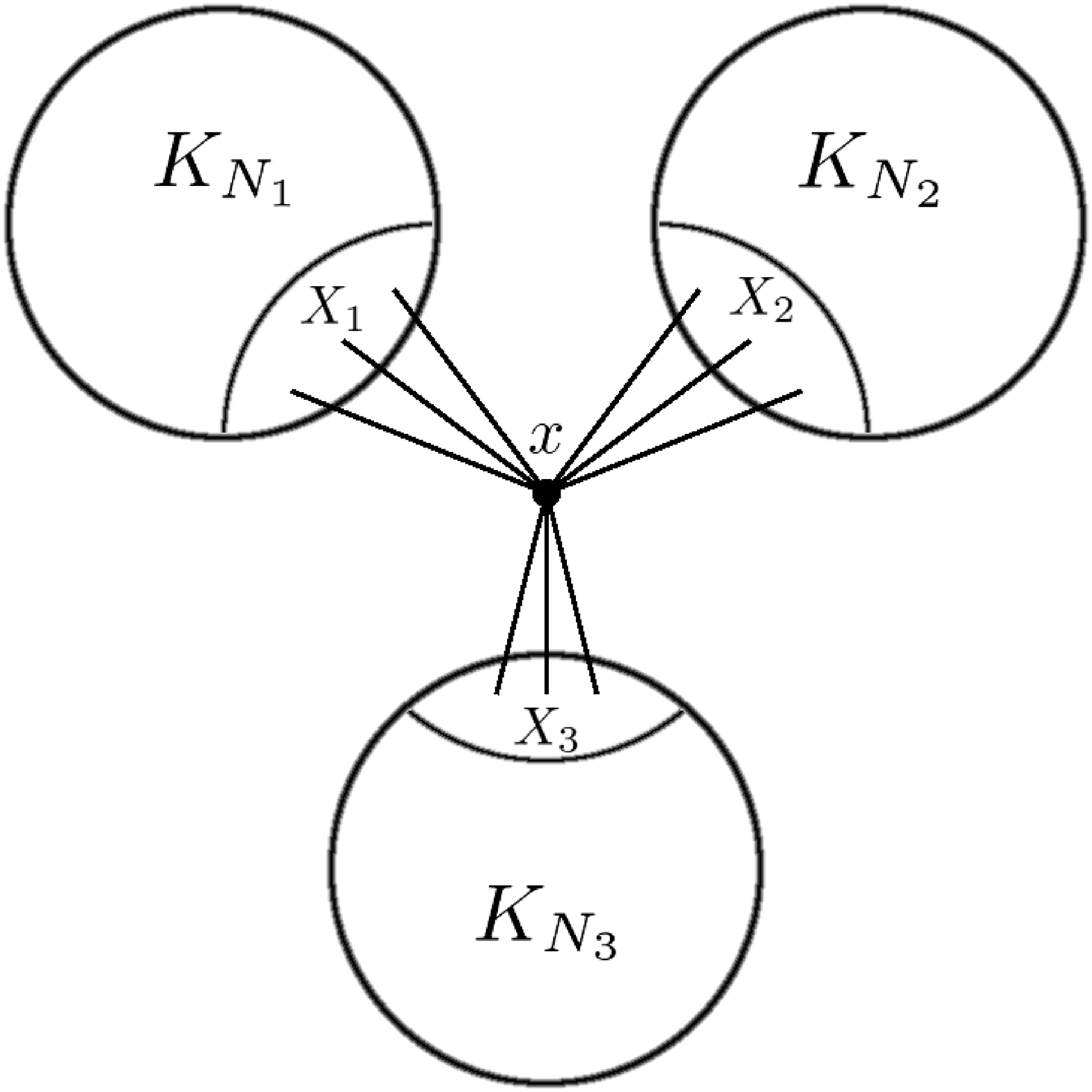}
\caption{A graph with a bad tangle-distinguishing tree-decomposition.}\label{f:example}
\end{figure}

We will show that, for the canonical tree-decompositions of Carmesin et al, the torsos of the inessential parts all have branch-width $<\!k$ and so it is possible to decompose the torsos of the inessential parts in this way.

\begin{theorem}\label{t:maintangle}
For every graph $G$ and $k \geq 3$ there exists a canonical tree-decompositon $(T,\cc{V})$ of $G$ of adhesion $<\!k$ such that
\begin{itemize}
\item $(T,\cc{V})$ distinguishes the set of $k$-tangles in $G$ efficiently;
\item The torso of every inessential part has branch-width $<\!k$.
\end{itemize}
\end{theorem}

More recently another potential candidate for these `$k$-connected components' has been considered in the literature, called \emph{$k$-blocks}. We say that a set of at least $k$ vertices in a graph is \emph{$(<\!k)$-inseparable} if no set of $<\!k$ vertices can separate any two of the vertices. A \emph{$k$-block} is a maximal $(<\!k)$-inseparable set of vertices. These $k$-blocks differ from subgraphs which are $k$-connected in the classical sense in that their connectivity is measured in the ambient graph rather than the subgraph itself. For example if we take a large independent set, $I$, and join each pair of vertices in $I$ by $k$ vertex disjoint paths, then $I$ is a $k$-block, even though as a subgraph it is independent. Carmesin, Diestel, Hundertmark and Stein \cite{confing} showed that, for any graph $G$, there is a canonical tree-decomposition which distinguishes the set of $k$-blocks. The work of Carmesin et al \cite{CDHH13CanonicalAlg} extended the results of \cite{confing} to more general types of highly connected substructures in graphs, and these results have been extended further by Diestel, Hundertmark and Lemanczyk \cite{HL2015} to more general combinatorial structures, such as matroids.

As before, these tree-decompositions will have some parts which are \emph{essential}, that is they contain a $k$-block, and some parts which are \emph{inessential}, and it is natural to ask about the structure of these parts. Recently, Diestel, Eberenz and Erde \cite{Block} proved a duality theorem for $k$-blocks, analogous to the tangle/branch-width duality of Robertson and Seymour. The result implies that a graph contains a $k$-block if and only if it does not admit a tree-decomposition of \emph{block-width} $<\!k$, where as before, every part in a tree-decomposition of block-width $<\!k$ is in some sense `too small' to contain a $k$-block. We also show a corresponding result for blocks.

\begin{theorem}\label{t:mainblock}
For every graph $G$ and $k \geq 3$ there exists a canonical tree-decompositon $(T,\cc{V})$ of $G$ of adhesion $<\!k$ such that
\begin{itemize}
\item $(T,\cc{V})$ distinguishes the set of $k$-blocks in $G$ efficiently;
\item The torso of every inessential part has block-width $<\!k$.
\end{itemize}
\end{theorem}

The main result in this paper, of which Theorems \ref{t:maintangle} and \ref{t:mainblock} are corollaries, is a lemma that gives sufficient conditions on the separators of an inessential part in a distinguishing tree-decomposition for the torso to have small width. These conditions seem quite natural and reasonable, in particular they are satisfied by every part of the canonical tangle/block-distinguishing tree-decompositions constructed by Carmesin et al. 

In some sense the canonical tangle-distinguishing tree-decompositions tell us most about the structure of the graph when the essential parts correspond closely to the tangles inside them. For example consider the following two graphs, firstly two $K_{N}$s overlapping in $k-1$ vertices and secondly two $K_{3k/2}$s each with a long path attached, of length $N' = N - 3k/2$, overlapping in a similar way, see Figure \ref{f:inessential}. 

\begin{figure}[!ht]
\center
\includegraphics[scale=0.15]{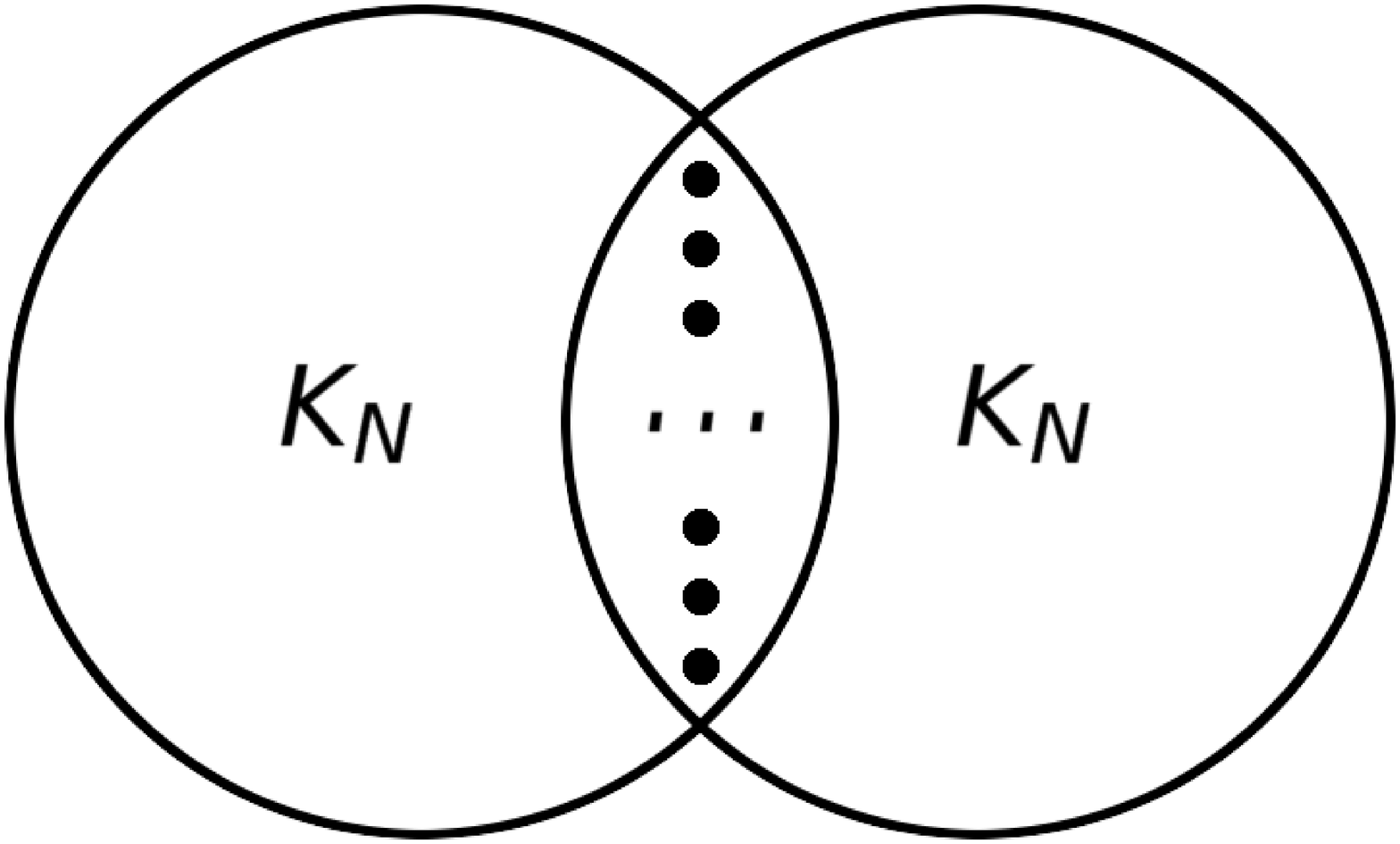}\hspace{1cm}
\includegraphics[scale=0.15]{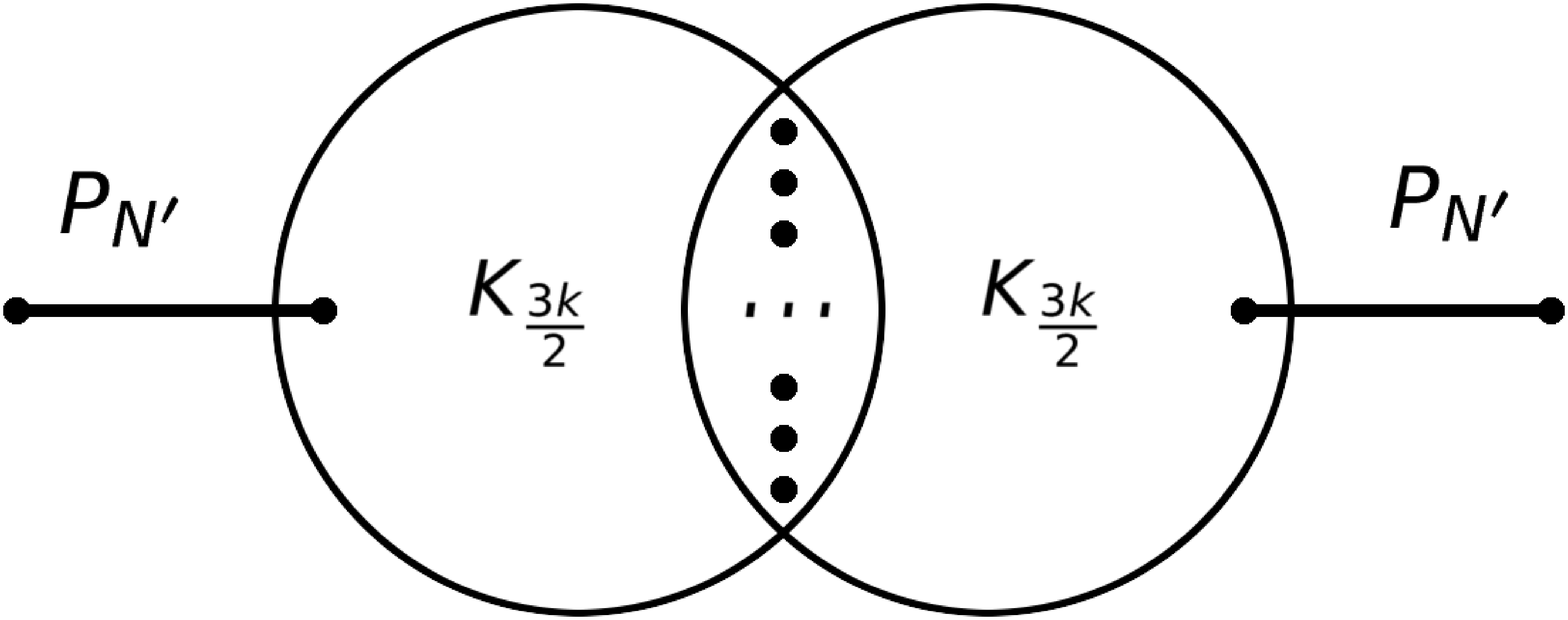}
\caption{Two graphs with the same canonical $k$-tangle-distinguishing tree-decomposition.}\label{f:inessential}
\end{figure}

Since the tangle-distinguishing tree-decompositions of Carmesin et al. only use essential separations, that is separations which distinguish some pair of $k$-tangles, they will construct the same tree-decomposition for both of these graphs, with just two parts of size $N$. However in the second example a more sensible tree-decomposition would further split up the long paths. This could be done in a way to maintain the property that the inessential parts have small branch-width, and by separating these inessential parts from the essential part we have more precisely exhibited the structure of the graph. We will also apply our methods to the problem of further refining the essential parts of these tree-decompositions.

In Section \ref{s:back} we introduce the background material necessary for our proof and in Section \ref{s:ref} we prove our central lemma and deduce the main results in the paper. In Section \ref{s:further} we discuss how our methods can also be used to further refine the essential parts of a tree-decomposition.

\section{Background material}\label{s:back}
\subsection{Separation systems and tree-decompositions}
A \emph{separation} of a graph $G$ is a set $\{A,B\}$ of subsets of $V(G)$ such that $A \cup B = V$ and there is no edge of $G$ between $A \setminus B$ and $B \setminus A$. There are two \emph{oriented separations} associated with a separation, $(A,B)$ and $(B,A)$. Informally we think of $(A,B)$ as \emph{pointing towards} $B$ and \emph{away from} $A$. We can define a partial ordering on the set of oriented separations of $G$ by 
\[
(A,B) \leq (C,D) \text{ if and only if } A \subseteq C \text{ and } B \supseteq D.
\]
The \emph{inverse} of an oriented separation $(A,B)$ is the separation $(B,A)$, and we note that mapping every oriented separation to its inverse is an involution which reverses the partial ordering.

In \cite{DiestelOumDualityIa} Diestel and Oum generalised these properties of separations of graphs and worked in a more abstract setting. They defined a \emph{separation system} $(\ra{S},\leq,*)$ to be a partially ordered set $\ra{S}$ with an order reversing involution, $*$. The elements of $\ra{S}$ are called \emph{oriented separations}. Often a given element of $\ra{S}$ is denoted by $\ra{s}$, in which case its inverse $\ra{s}^*$ will be denoted by $\la{s}$, and vice versa. Since $*$ is ordering reversing we have that, for all $\ra{r},\ra{s} \in S$,
\[
\ra{r} \leq \ra{s} \text{ if and only if } \la{r} \geq \la{s}.
\]
A \emph{separation} is a set of the form $\{\ra{s},\la{s}\}$, and will be denoted by simply $s$. The two elements $\ra{s}$ and $\la{s}$ are the \emph{orientations} of $s$. The set of all such pairs $\{\ra{s},\la{s}\} \subseteq \ra{S}$ will be denoted by $S$.  If $\ra{s}=\la{s}$ we say $s$ is \emph{degenerate}. Conversely, given a set $S' \subseteq S$ of separations we write $\ra{S'} := \bigcup S'$ for the set of all orientations of its elements. With the ordering and involution induced from $\ra{S}$, this will form a separation system. When we refer to a oriented separation in a context where the notation explicitly indicates orientation, such as $\ra{s}$ or $(A,B)$, we will usually suppress the prefix ``oriented" to improve the flow of the paper.

Given a separation of a graph $\{A,B\}$ we can identify it with the pair $\{(A,B),(B,A)\}$ and in this way any set of separations in a graph which is closed under taking inverses forms a separation system. We will work within the framework developed in \cite{DiestelOumDualityIa} since we will need to use directly some results proved in this abstract setting, but also because our results are most easily expressible in this framework. An effort has been made to state the results in the widest generality, so as to be applicable in the broadest sense, however we will always have in mind the motivating example of separation systems which arise as sets of separations in a graph, and so a reader will not lose too much by thinking about these separation systems solely in those terms. 

The \emph{separator} of a separation $\ra{s}=(A,B)$ in a graph is the intersection $A \cap B$ and the \emph{order} of a separation, $|\ra{s}|=$ ord$(A,B)$, is the cardinality of the separator $|A \cap B|$.  Note that if $\ra{r}=(A,B)$ and $\ra{s}=(C,D)$ are separations then so are the \emph{corner separations} $\ra{r} \vee \ra{s} := (A \cup C, B \cap D)$ and $\ra{r} \wedge \ra{s} := (A \cap C, , B \cup D)$ and the orders of these separations satisfy the equality 
\[
 |\ra{r} \vee \ra{s}| + |\ra{r} \wedge \ra{s}| = |\ra{r}| + |\ra{s}|.
\]
Hence the order function is a submodular function on the set of separations of a graph, and we note also that it is clearly symmetric.

\begin{figure}[!ht]
\centering
\includegraphics[scale=0.25]{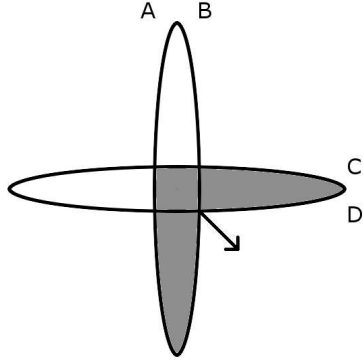}
\caption{Two separations $(A,B)$ and $(C,D)$ with the corner separation $(A \cup C, B \cap D)$ marked.}
\end{figure}

For abstract separations systems, if there exists binary operations $\vee$ and $\wedge$ on $\ra{S}$ such that $\ra{r} \vee \ra{s}$ is the supremum and $\ra{r} \wedge \ra{s}$ is the infimum of $\ra{r}$ and $\ra{s}$ then we call $(\ra{S},\leq,*,\vee,\wedge)$ a \emph{universe} of (oriented) separations, and we call any real, non-negative, symmetric and submodular function on a universe an \emph{order function}. 

Two separations $r$ and $s$ are \emph{nested} if they have $\leq$-comparable orientations. Two oriented separations $\ra{r}$ and $\ra{s}$ are \emph{nested} if $r$ and $s$ are nested \footnote{In general we will use terms defined for separations informally for oriented separations when the meaning is clear, and vice versa}. If $\ra{r}$ and $\ra{s}$ are not nested we say that the two separations \emph{cross}. A set of separations $S$ is \emph{nested} if every pair of separations in $S$ is nested, and a separation $s$ is \emph{nested} with a set of separations $S$ if $S \cup \{s\}$ is nested. 

A separation $\ra{r} \in \ra{S}$ is \emph{trivial in $\ra{S}$}, and $\la{r}$ is \emph{co-trivial}, if there exist an $s \in S$ such that $\ra{r} < \ra{s}$ and $\ra{r} < \la{s}$.  Note that if $\ra{r}$ is trivial, witnessed by some $s$, then, since the involution is order reversing, we have that $\ra{r} < \ra{s} < \la{r}$. So, in particular, $\la{r}$ cannot also be trivial. Separations $\ra{s}$ such that $\ra{s} \leq \la{s}$, trivial or not, will be called \emph{small} and their inverses \emph{co-small}.

In the case of separations of a graph, it is a simple check that the small separations are precisely those of the form $(A,V)$. Furthermore the trivial separations can be characterised as those of the form $(A,V)$ such that $A \subseteq C \cap D$ for some separation $(C,D)$ such that $\{C,D\} \neq \{A,B\}$. Finally we note that there is only one degenerate separation in a graph, $(V,V)$.

A \emph{tree-decomposition} of a graph $G$ is a pair $(T,\cc{V})$ consisting of a tree $T$ and family $\cc{V} = (V_t)_{t \in T}$ of vertex sets $V_t \subseteq V(G)$, one for each vertex $t \in T$ such that:
\begin{itemize}
\item $V(G) = \bigcup_{t \in T} V_t $;
\item for every edge $e \in G$ there exists some $t \in T$ such that $e \in G[V_t]$;
\item $V_{t_1} \cap V_{t_2} \subseteq V_{t_3}$  whenever $t_3$ lies on the $t_1-t_2$ path in $T$.
\end{itemize}

The sets $V_t$ in a tree-decomposition are its \emph{parts} and the sets $V_t \cap V_{t'}$ such that $(t,t')$ is an edge of $T$ are the \emph{adhesion sets}. The \emph{torso} of a part $\overline{V_t}$ is the union of that part together with the completion of the adhesion sets adjacent to that part, that is
\[
\overline{V_t} = G|_{V_t} \cup \bigcup_{(t,t') \in T} K_{V_t \cap V_{t'}}.
\]
The \emph{width} of a tree-decomposition is $\max \{ |V_t| - 1 \, : \, \text{ such that } t \in T\}$, and the \emph{adhesion} is the size of the largest adhesion set. Deleting an oriented edge $e=(t_1,t_2) \in \ra{E}(T)$ divides $T-e$ into two components $T_1 \ni t_1$ and $T_2 \ni t_2$. Then $(\bigcup_{t \in T_1} V_t, \bigcup_{t \in T_2} V_t)$ can be seen to be a separation of $G$ with separator $V_{t_1} \cap V_{t_2}$. We say that the edge $e$ \emph{induces} this separation. Given a tree-decomposition $(T,\cc{V})$ it is easy to check that the set of separations induced by the edges of $T$ form a nested separation system. Conversely it was shown in \cite{confing} that every nested separation system is induced by some tree-decomposition, and so in a sense these two concepts can be thought of as equivalent. 

We say that a nested set of separations $\cc{N'}$ \emph{refines} a nested set of separations $\cc{N}$ if $\cc{N'} \supseteq \cc{N}$, and similarly a tree-decomposition $(T',\cc{V}')$ \emph{refines} a tree-decomposition $(T,\cc{V})$ if the set of separations induced by the edges of $T'$ refines the corresponding set of separations for $T$.

\subsection{Duality of tree-decompositions}\label{s:dual}
There are a number of theorems that assert a duality between certain structurally `large' objects in a graph and an overall tree structure. For example a graph has small tree-width if and only if it contains no large order bramble \cite{ST1993GraphSearching}. In \cite{DiestelOumDualityIa} a general theory of duality, in terms of separation systems, was developed which implied many of the existing theorems. Following on from the notion of tangles in graph minor theory \cite{GMX} these large objects were described as orientations of separations systems avoiding certain forbidden subsets. 

An \emph{orientation} of a set of separations $S$ is a subset $O \subseteq \ra{S}$ which for each $s \in S$ contains exactly one of its orientations $\ra{s}$ or $\la{s}$. A \emph{partial orientation} of $S$ is an orientation of some subset of $S$, and we say that an orientation $O$ \emph{extends} a partial orientation $P$ if $P \subseteq O$. 

In our context we will think of an orientation $O$ on some set of graph separations as choosing a side of each separation $s=\{A,B\}$ to designate as large. For example given a graph $G$ and the set $S$ of all separations of the graph $G$, we denote by 
\[
{\ra{S}\!}_k = \{ \ra{s} \in \ra{S} \,:\, |\ra{s}| < k\},
\]
the set of all orientations of order less than $k$. If there is a large clique (of size $\geq k$) in $G$ then for every $s = \{A,B\} \in S_k$ we have that the clique is contained entirely in $A$ or $B$. So this clique defines an orientation of $S_k$ by picking, for each $\{A,B\} \in S_k$ the orientated separation such that the clique is contained in second set in the pair. 

We call an orientation $O$ of a set of separations $S$ \emph{consistent} if whenever we have distinct $r$ and $s$ such that $\ra{r} < \ra{s}$, $O$ does not contain both $\la{r}$ and $\ra{s}$. Note that a consistent orientation must contain all trivial separations $\ra{r}$, since if $\ra{r} < \ra{s}$ and $\ra{r} < \la{s}$ then, whichever orientation of $s$ is contained in $O$ would be inconsistent with $\la{r}$.

Given a set of subsets $\cc{F} \subseteq 2^{\ra{S}}$ we say that an orientation $O$ is \emph{$\cc{F}$-avoiding} if there is no $F \in \cc{F}$ such that $F \subseteq O$. So for example an orientation is consistent if it avoids $\cc{F} = \{ \{\la{r},\ra{s}\} \, : \, r \neq s, \ra{r} < \ra{s} \}$. In general we will define the `large' objects we consider by the collection $\cc{F}$ of subsets they avoid. For example a $k$-tangle in a graph $G$ can easily be seen to be equivalent to an orientation of $S_k$ which avoids the set of triples 
\[
\cc{T}_k = \{\{(A_1,B_1),(A_2,B_2),(A_3,B_3)\} \subseteq {\ra{S}\!}_k \,: \, \bigcup_{i=1}^3 G[A_i] = G\}.
\]
(Where the three separations need not be distinct). That is, a tangle is an orientation such that no three small sides cover the entire graph, it is a simple check that any such orientation must in fact also be consistent. We say that a consistent orientation which avoids a set $\cc{F}$ is an \emph{$\cc{F}$-tangle}.

Given a set $\cc{F} \subseteq 2^{\ra{S}}$, an \emph{$S$-tree over $\cc{F}$} is a pair $(T,\alpha)$, of a tree $T$ with at least one edge and a function $\alpha : \ra{E}(T) \rightarrow \ra{S}$ from the set 
\[
\ra{E}(T) := \{ (x,y) \,: \, \{x,y\} \in E(T) \}
\]
of orientations of it's edges to $\ra{S}$ such that:
\begin{itemize}
\item For each edge $(t_1,t_2) \in \ra{E}(T) $, if $\alpha(t_1,t_2) = \ra{s}$ then $\alpha(t_2,t_1) = \la{s}$;
\item For each vertex $t \in T$, the set $\{\alpha(t',t) \, :\, (t',t) \in \ra{E}(T) \}$ is in $\cc{F}$;
\end{itemize}

For any leaf vertex $w \in T$ which is adjacent to some vertex $u \in T$ we call the separation $\ra{s}=\alpha(w,u)$ a \emph{leaf separation} of $(T,\alpha)$. A particularly interesting class of such trees is when the set $\cc{F}$ is chosen to consist of stars. A set of non-degenerate oriented separations $\sigma$ is called a \emph{star} if $\ra{r} \leq \la{s}$ for all distinct $\ra{r},\ra{s} \in \sigma$. In what follows, if we refer to an $S$-tree without reference to a specific family $\cc{F}$ of stars, it can be assumed to be over the set of \emph{all} stars in $2^{\ra{S}}$. We say that an $S$-tree over $\cc{F}$ is \emph{irredundant} if there is no $t \in T$ with two neighbours, $t'$ and $t''$ such that $\alpha(t,t') = \alpha(t,t'')$. If $(T,\alpha)$ is an irredundant $S$-tree over a set of stars $\cc{F}$, then it is easy to verify that the map $\alpha$ preserves the natural ordering on $\ra{E}(T)$, defined by letting $(s,t) \leq (u,v)$ if the unique path in $T$ between those edges starts at $s$ and ends at $v$ (see [\cite{DiestelOumDualityIa}, Lemma 2.2]).

Given an irredundant $S$-tree $(T,\alpha)$ over a set of stars and an orientation $O$ of $S$, $O$ induces an orientation of the edges of $T$, which will necessarily contain a sink vertex. If the orientation $O$ is consistent then this sink vertex, which we will denote by $t$, will be unique. We say that $O$ is \emph{contained in $t$}. If $S=S_k$ for some graph $G$, we have that $(T,\alpha)$ defines some tree-decomposition $(T,\cc{V})$ of $G$, and we say that $O$ is \emph{contained in the part $V_t$}. So, each $\cc{F}$-tangle of $S$ must live in some vertex of every such $S$-tree, and by definition this vertex give rise to a star of separations in $\cc{F}$. In this way, each of the vertices in an $S$-tree over $\cc{F}$ (and each of the parts in the corresponding tree-decomposition when one exists) is `too small' to contain an $\cc{F}$-tangle.

Suppose we have a separation $\ra{r}$ which is neither trivial nor degenerate. In applications $\ra{r}$ will be a leaf separation in some irredundant $S$-tree over a set $\cc{F}$ of stars. Given some $\ra{s} \geq \ra{r}$, it will be useful to have a procedure to `shift' the $S$-tree $(T,\alpha)$ in which $\ra{r}$ is a leaf separation to a new $S$-tree $(T,\alpha')$ such that $\ra{s}$ is a leaf separation. Let $S_{\geq \ra{r}}$ be the set of separations $x \in S$ that have an orientation $\ra{x} \geq \ra{r}$. Since $\ra{r}$ is a leaf separation in an irredundant $S$-tree over a set of stars we have by the previous comments that the image of $\alpha$ is contained in $\ra{S}_{\geq \ra{r}}$.

Given $x \in S_{\geq \ra{r}} \setminus \{r\}$ we have, since $\ra{r}$ is non-trivial, that only one of the two orientations of $x$, say $\ra{x}$ is such that $\ra{x} \geq \ra{r}$. So, we can define a function $f\downarrow^{\ra{r}}_{\ra{s}}$ on $\ra{S}_{\geq \ra{r}} \setminus \{\la{r}\}$ by\footnote{The exclusion of $\la{r}$ here is for a technical reason, since it could be the case that $\ra{r} < \la{r}$, however we want to insist that $f\downarrow^{\ra{r}}_{\ra{s}}(\la{r})$ is the inverse of $f\downarrow^{\ra{r}}_{\ra{s}}(\ra{r})$ }
\[
f\downarrow^{\ra{r}}_{\ra{s}}(\ra{x}) := \ra{x} \vee \ra{s} \text{ and } f\downarrow^{\ra{r}}_{\ra{s}}(\la{x}) := (\ra{x} \vee \ra{s})^*.
\]

\begin{figure}[!ht]
\centering
\includegraphics[scale=0.2]{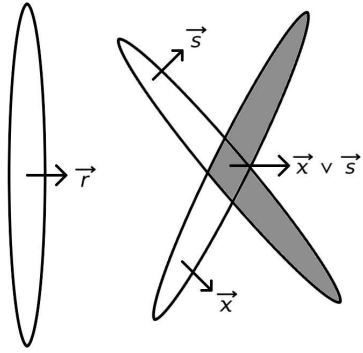}
\caption{Shifting a separation $\ra{x} \geq \ra{r}$ under $f\downarrow^{\ra{r}}_{\ra{s}}$. }
\end{figure}

Given an $S$-tree $(T,\alpha)$ and $\ra{s} \geq \ra{r}$ as above let $\alpha':= f\downarrow^{\ra{r}}_{\ra{s}}\circ\, \alpha$. The \emph{shift of $(T,\alpha)$ onto $\ra{s}$} is the $S$-tree $(T,\alpha')$.

We say that $\ra{s}$ \emph{emulates $\ra{r}$ in $\ra{S}$} if $\ra{r} \leq \ra{s}$ and for every $\ra{t} \in \ra{S}_{\geq \ra{r}} \setminus \{\la{r}\}$, $\ra{s} \vee \ra{t} \in \ra{S}$. Given a particular set of stars $\cc{F} \subseteq 2^{\ra{S}}$ we say further that $\ra{s}$ \emph{emulates $\ra{r}$ in $\ra{S}$ for $\cc{F}$} if $\ra{s}$ emulates $\ra{r}$ in $\ra{S}$ and for any star $\sigma \subset \ra{S}_{\geq \ra{r}} \setminus \{\la{r}\}$ in $\cc{F}$ that contains an element $\ra{t} \geq \ra{r}$ we also have $f\downarrow^{\ra{r}}_{\ra{s}}(\sigma) \in \cc{F}$. The usefulness of this property is exhibited by the following lemma, which is key both in the proof of Theorem \ref{t:dual} from \cite{DiestelOumDualityIa}, and will be essential for the proof of our central lemma.

\begin{lemma}\label{l:link}[\cite{DiestelOumDualityIa}, Lemma 4.2]
Let $(\ra{S},\leq,^*)$ be a separation system, $\cc{F} \subseteq 2^{\ra{S}}$ a set of stars, and let $(T,\alpha)$ be an irredundant $S$-tree over $\cc{F}$. Let $\ra{r}$ be a nontrivial and nondegenerate separation which is a leaf separation of $(T,\alpha)$, and is not the image of any other edge in $T$, and let $\ra{s}$ emulate $\ra{r}$ in $\ra{S}$. Then the shift of $(T,\alpha)$ onto $\ra{s}$ is an $S$-tree over $\cc{F} \cup \{\{\la{s}\}\}$ in which $\ra{s}$ is a leaf separation, associated with a unique leaf.
\end{lemma}

It is shown in [\cite{DiestelOumDualityIa}, Lemma 2.4] that if we have an $S$-tree over $\cc{F}$, $(T,\alpha)$, and a set of non-trivial and non-degenerate leaf separations, ${\ra{r}\!}_i$, of $(T,\alpha)$ then there also exists an irredundant $S$-tree over $\cc{F}$, $(T',\alpha')$, such that each ${\ra{r}\!}_i$ is a leaf separation of $(T',\alpha)$ and is not the image of any other edge in $T'$.

We say a set $\cc{F} \subseteq 2^{\ra{S}}$ \emph{forces} a separation $\ra{r}$ if $\{ \la{r} \} \in \cc{F}$ or $r$ is degenerate. Note that the non-degenerate forced separations in $\cc{F}$ are precisely those separations which can appear as leaf separations in an $S$-tree over $\cc{F}$. We say $\cc{F}$ is \emph{standard} if it forces every trivial separation in $\ra{S}$. 

We say that a separation system $\ra{S}$ is \emph{separable} if for any two non-trivial and non-degenerate separations $\ra{r},\la{r}' \in \ra{S}$ such that $\ra{r} \leq \ra{r}'$ there exists a separation $s \in S$ such that $\ra{s}$ emulates $\ra{r}$ in $\ra{S}$ and $\la{s}$ emulates $\la{r}'$ in $\ra{S}$. We say that $\ra{S}$ is \emph{$\cc{F}$-separable} if for all non-trivial and non-degenerate $\ra{r},\la{r}' \in \ra{S}$ that are not forced by $\cc{F}$ such that $\ra{r} \leq \ra{r}'$ there exists a separation $s \in S$ such that $\ra{s}$ emulates $\ra{r}$ in $\ra{S}$ for $\cc{F}$ and $\la{s}$ emulates $\la{r}'$ in $\ra{S}$ for $\cc{F}$. Often one proves that $\ra{S}$ is $\cc{F}$-separable in two steps, first by showing it is separable, and then by showing that $\cc{F}$ is \emph{closed under shifting}: that whenever $\ra{s}$ emulates some $\ra{r}$ in $\ra{S}$, it also emulates that $\ra{r}$ in $\ra{S}$ for $\cc{F}$. 

We are now in a position to state the Strong Duality Theorem from \cite{DiestelOumDualityIa}.

\begin{theorem}\label{t:dual}[\cite{DiestelOumDualityIa}, Theorem 4.3]
Let $(\ra{U},\leq,*,\vee,\wedge)$ be a universe of separations containing a separation system $(\ra{S},\leq, *)$. Let $\cc{F} \subseteq 2^{\ra{S}}$ be a standard set of stars. If $\ra{S}$ is $\cc{F}$-separable, exactly one of the following assertions holds:
\begin{itemize}
\item There exists an $S$-tree over $\cc{F}$. 
\item There exists an $\cc{F}$-tangle of $S$.
\end{itemize}
\end{theorem}

The property of being $\cc{F}$-separable may seem a rather strong condition to hold, however in \cite{DiestelOumDualityIb} it is shown that for all the sets $\cc{F}$ describing classical `large' objects (such as tangles or brambles) the separation systems ${\ra{S}\!}_k$ are $\cc{F}$-separable. More specifically, by definition a $k$-tangle is a consistent orientation which avoids the set $\cc{T}_k$ as defined earlier. In fact it is shown in \cite{DiestelOumDualityIb} that a consistent orientation avoids $\cc{T}_k$ if and only if it avoids the set of stars in $\cc{T}_k$
\[
\cc{T}_k^* = \{ \{(A_i,B_i)\}_1^3 \,:\, \{(A_i,B_i)\}_1^3 \subseteq S_k \text{ is a star and } \bigcup_i G[A_i] = G \}.
\]
Note that $\cc{T}_k^*$ is standard. Indeed it forces all the small separations $(A,V)$, and so it forces the trivial separations. It can also be checked that ${\ra{S}\!}_k$ is $\cc{T}_k^*$-separable.

The dual structure to a $k$-tangle is therefore an $S_k$-tree over $\cc{T}_k^*$. It is shown in \cite{DiestelOumDualityIb} that the existence of such an $S_k$-tree is equivalent to the existence of a branch-decomposition of width $<\!k$ for all $k\geq 3$. We note that the condition that $k \geq 3$ is due to a quirk in how branch-width is traditionally defined, which results in, for example, stars having branch-width $1$ but all other trees having branch-width $2$, whilst both contain $2$-tangles.

If a tree-decomposition $(T,\cc{V})$ of a graph $G$ is such that the set of separations induced by the edges of $T$ is an $S_k$-tree over $\cc{T}_k^*$ for some $k$, then there is some smallest such $k'$, and we say the \emph{branch-width} of the tree-decomposition is $k'-1$. If no such $k$ exists then we will let the branch-width be infinite. By the preceding discussion we have that the branch-width (in the traditional sense) of a graph is the smallest $k$ such that $G$ has a tree-decomposition of branch-width $k$ (except when the branch-width of $G$ is $1$), and so this should not cause too much confusion.

\subsection{Canonical tree-Decompositions distinguishing tangles}
Given two orientations $O_1$ and $O_2$ of a set of separations $S$ we say that a separation $s$ \emph{distinguishes} $O_1$ and $O_2$ if $\ra{s} \in O_1$ and $\la{s} \in O_2$. As in the previous section, every tree-decomposition,  $(T,\cc{V})$, corresponds to some nested set of separations, $\cc{N}$. We say that a tree-decomposition \emph{distinguishes} $O_1$ and $O_2$ if there is some separation in $\cc{N}$ which distinguishes $O_1$ and $O_2$. If $O_1$ and $O_2$ are consistent, then the tree-decomposition will distinguish them if and only if they are contained in different parts of the tree. 

As in Section \ref{s:dual} a $k$-block $b$ can be viewed as an orientation of $S_k$. Indeed given any separation $(A,B)$ with ord$(A,B) < k$, since $b$ is $(<\!k)$-inseparable, $b \subseteq A$ or $b \subseteq B$, so we can think of $b$ as orienting each $s \in S_k$ towards the side of the separations that $b$ lies in. In \cite{confing} Carmesin, Diestel, Hundertmark and Stein showed how to algorithmically construct a nested set of separations in a graph $G$ (and so a tree-decomposition) in a canonical way, that is, invariant with respect to the automorphism group of $G$, which distinguishes all of its $k$-blocks, for a given $k$. 

These ideas were extended in \cite{CDHH13CanonicalAlg} to construct canonical tree-decompositions which distinguish all the $k$-profiles in a graph, a common generalization of $k$-tangles and $k$-blocks. A \emph{$k$-profile} can be defined as a $\cc{P}_k$-tangle of $S_k$, where
\[
\cc{P}_k = \{ \sigma = \{ (A,B),(C,D),(B \cap D, A \cup C) \} \,: \, \sigma \subseteq {\ra{S}\!}_k\}.
\]

More generally, given a universe of separations $(\ra{U},\leq,*,\vee,\wedge)$ with an order function containing a separation system $(\ra{S},\leq, *)$, we can define as before an \emph{$S$-profile} to be a $\cc{P}_S$-tangle of $S$ where
\[
\cc{P}_S = \{ \sigma = \{ \ra{r},\ra{s}, \la{r} \wedge \la{s} \} \,: \, \sigma \subseteq \ra{S}\}.
\]
Given two distinct $S$-profiles $P_1$ and $P_2$ there is some $s \in S$ which distinguishes them. Furthermore, there is some minimal $l$ such that there is a separation of order $l$ which distinguishes $P_1$ and $P_2$, and we define $\kappa(P_1,P_2) :=l$. We say that a separation $s$ distinguishes $P_1$ and $P_2$ \emph{efficiently} if $s$ distinguishes $P_1$ and $P_2$ and $|s| = \kappa(P_1,P_2)$. Given a set of profiles $\phi$ we say that a separation $s$ is \emph{$\phi$-essential} if it efficiently distinguishes some pair of profiles in $\phi$. We will often consider in particular, as in the case of graphs, the separation system arising from those separations in a universe of order $<\!k$, that is we define
\[
{\ra{S}\!}_k = \{ \ra{u} \in \ra{U} \,: \, |\ra{u}| < k \},
\]
where in general it should be clear from the context which universe $S_k$ lives in.

In \cite{CDHH13CanonicalAlg} a number of different algorithms, which they call \emph{$k$-strategies}, are described for constructing a nested set of separations distinguishing a set of profiles. These algorithms build the set of separations in a series of steps, and at each step there is a number of options for how to pick the next set of separations. A $k$-strategy is then a description of which choice to make at each step. The authors showed that, regardless of which choices are made at each step, this algorithm will produce a nested set of separations distinguishing all the profiles in $G$. We say a set of profiles is \emph{canonical} if it is fixed under every automorphism of $G$. In particular the following is shown.

\begin{theorem}\label{t:dist}[\cite{CDHH13CanonicalAlg} Theorem 4.4]
Every $k$-strategy $\Sigma$ determines for every canonical set $\phi$ of $k$-profiles of a graph $G$ a canonical nested set $\cc{N}_{\Sigma}(G,\phi)$ of $\phi$-essential separations of order $<\!k$ that distinguishes all the profiles in $\phi$ efficiently.
\end{theorem}

Note that any $k$-tangle, $O$, is also a $k$-profile. Indeed, it is a simple check that $O$ is consistent. Also for any pair of separations $(A,B),(C,D) \in {\ra{S}\!}_k$ we have that $G[A] \cup G[C] \cup G[B \cap D] = G$, since any edge not contained in $A$ or $C$ is contained in both $B$ and $D$. Hence, $\{(A,B),(C,D),(B \cap D, A \cup C)\} \in \cc{T}_k$, and so $\cc{P}_k \subseteq \cc{T}_k$. Therefore any $k$-tangle, which by definition avoids $\cc{T}_k$, must also avoid $\cc{P}_k$, and so must be a $k$-profile. Similarly one can show that the orientations defined by $k$-blocks are consistent and $\cc{P}_k$ avoiding, and so $k$-profiles. Even more, there is some family $\cc{B}_k \supseteq \cc{P}_k$ such that the orientations defined by $k$-blocks are $\cc{B}_k$-tangles, and if there is a $\cc{B}_k$-tangle of $S_k$ then the graph $G$ contains a unique $k$-block corresponding to this orientation.

One of the aims of \cite{DiestelOumDualityIa,DiestelOumDualityIb} had been to develop a duality theorem which would be applicable to $k$-profiles and $k$-blocks. The same authors showed in \cite{DiestelOumDualityII} that there is a more general duality theorem of a similar kind which applies in these cases, however the dual objects in this theorem correspond to a more general object than the classical notion of tree-decompositions.

Nevertheless, it was posed as an open question whether or not there was a duality theorem for $k$-profiles or $k$-blocks expressible within the framework of \cite{DiestelOumDualityIa}. By Theorem \ref{t:dual} it would be sufficient to show that there is a standard set of stars $\cc{F}$ such that the set of $k$-profiles or $k$-blocks coincides with the set of $\cc{F}$-tangles. Recently Diestel, Eberenz and Erde \cite{Block} showed that, if we insist the orientations satisfy a slightly stronger consistency condition, this will be the case. We say that an orientation $O$ of a separation system $S$ is \emph{regular} if whenever we have $r$ and $s$ such that $\ra{r} \leq \ra{s}$, $O$ does not contain both $\la{r}$ and $\ra{s}$. We note that a consistent orientation is regular if and only if it contains every small separation. A \emph{regular $\cc{F}$-tangle of $S$} is then a regular $\cc{F}$-avoiding orientation of $S$, and a \emph{regular $S$-profile} is a regular $\cc{P}_S$-tangle. For most natural examples of separation systems there will not be a difference between regular and irregular profiles. Indeed, in \cite{Block} it is shown that for $k \geq 3$ every $k$-profile of a graph is in fact a regular $k$-profile\footnote{There do exist pathological examples of $2$-profiles in graphs which are not regular, however they can be easily characterized.}. 

We say a separation system is \emph{submodular} if whenever $\ra{r},\ra{s} \in \ra{S}$ either $\ra{r} \wedge \ra{s}$ or $\ra{r} \vee \ra{s} \in \ra{S}$. Note that, if a universe $U$ has an order function, then the separation systems $S_k$ are submodular.

\begin{theorem}\label{t:genprof}[Diestel, Eberenz and Erde \cite{Block}]
Let $S$ be a separable submodular separation system contained in some universe of separations $(\ra{U},\leq,*,\vee,\wedge)$, and let $\cc{F} \supseteq \cc{P}_S$. Then there exists a standard set of stars $\cc{F}^*$ (which is closed under shifting, and contains $\{\ra{r}\}$ for every co-small $\ra{r}$) such that every regular $\cc{F}$-tangle of $S$ is an $\cc{F}^*$-tangle of $S$, and vice versa, and such that the following are equivalent:
\begin{itemize}
\item There is no regular $\cc{F}$-tangle of $S$;
\item There is no  $\cc{F}^*$-tangle of $S$;
\item There is an $S$-tree over $\cc{F}^*$.
\end{itemize}
\end{theorem}

In the case where $S=S_k$ is the set of separations of a graph with $k\geq 3$, we have that $\cc{F} \supseteq \cc{P}_k$,  and so every $\cc{F}$-tangle is a $k$-profile, and so regular. Hence, in this case, we can omit the word regular from the statement of the theorem. We note that $\cc{T}_k \supseteq \cc{P}_k$, (and in fact the $\cc{T}_k^*$ of the theorem can be taken to be the $\cc{T}_k^*$ defined earlier) and so Theorem \ref{t:genprof} also implies the tangle/branch-width duality theorem. 

Applying the result to $\cc{P}_k$ or $\cc{B}_k$ also gives a duality theorem for $k$-blocks and $k$-profiles. As in the case of tangles, if a tree-decomposition $(T,\cc{V})$ of a graph $G$ is such that the set of separations induced by the edges of $T$ is an $S_k$-tree over $\cc{P}_k^*$ for some $k$, then there is some smallest such $k'$, and we say the \emph{profile-width} of the tree-decomposition is $k'-1$. If no such $k$ exists then we will let the profile-width be infinite. The \emph{profile-width} of a graph is then the smallest $k$ such that $G$ has a tree-decompositions of profile-width $k$. Then, as was the case with tangles, Theorem \ref{t:genprof} tells us that the profile-width of a graph is the largest $k$ such that $G$ contains a $k$-profile. We define the \emph{block-width} of a tree-decomposition and graph in the same way.

In a similar way as before, we can think of any part in a tree-decomposition of block-width at most $k-1$ as being `too small' to contain a $k$-block, as the corresponding star of separations must lie in $\cc{B}_k^*$, and by Theorem \ref{t:genprof} every $k$-block defines an orientation of $S_k$ which avoids $\cc{B}_k^*$.

\section{Refining a tree-decomposition}\label{s:ref}
Given a set of profiles of a graph, $\phi$, we say a part $V_t$ of a tree-decomposition is \emph{$\phi$-essential} if some profile from $\phi$ is contained in this part. We will keep in mind as a motivating example the case $\phi=\tau_k$, the set of $k$-tangles and, when the set of profiles considered is clear, we will refer to such parts simply as \emph{essential}. Conversely if no such profile is contained in the part we call it \emph{inessential}.  The main result of the paper can now be stated formally.

\begin{lemma}\label{l:main}
Let $(\ra{U},\leq,*,\vee,\wedge)$ be a universe of separations with an order function. Let $\phi$ be a set of $S_k$-profiles and let $\cc{F}$ be a standard set of stars which contains $\{\ra{r}\}$ for every co-small $\ra{r}$, and which is closed under shifting, such that $\phi$ is the set of $\cc{F}$-tangles.  Let $\sigma = \{ {\ra{s}\!}_i \, : \, i \in [n] \} \subseteq {\ra{S}\!}_k$ be a non-empty star of separations such that each $s_i$ is $\phi$-essential, and let $\cc{F}' = \cc{F} \cup \bigcup_1^n \{\la{s_i}\}$

Then either there is an  $\cc{F}'$-tangle of $S_k$, or there is an $S_k$-tree over $\cc{F}'$ in which each ${\ra{s}\!}_i$ appears as a leaf separation.
\end{lemma}

If we compare Lemma \ref{l:main} to Theorem \ref{t:dual}, we see that Lemma \ref{l:main} can be viewed in some way as a method of building a new duality theorem from an old one, by adding some singleton separations to our set $\cc{F}$. The restriction to considering only $S_k$-profiles rather than those of arbitrary separation systems $S$ contained in $U$ comes from the proof, where we need to use the submodularity of the order function to show that certain separations emulate others. It would be interesting to know if the result would still be true for any $S$ which is separable, or even any pair $\cc{F}$ and $S$ such that $S$ is $\cc{F}$-separable. The condition that $\cc{F}$ contains every co-small separation as a singleton is to ensure that the $\cc{F}$-tangles are regular $\cc{F}$-tangles, as we will need to use the slightly stronger consistency condition in the proof.

What does Lemma \ref{l:main} say in the case of $k$-tangles arising from graphs? Recall that $\tau_k$ is the set of $\cc{T}_k^*$-tangles, and that $\cc{T}_k^*$ is closed under shifting, and contains $\{\ra{r}\}$ for every co-small $\ra{r}$. Given a star $\sigma = \{ {\ra{s}\!}_i \, : \, i \in [n] \} \subseteq {\ra{S}\!}_k$ we note that a $\cc{T}_k^* \cup \bigcup_1^n \{{\la{s}\!}_i\}$-tangle is just a $\cc{T}_k^*$-tangle which contains ${\ra{s}\!}_i$ for each $i$, and so it is a $k$-tangle which orients the star inwards. Conversely, an $S_k$-tree over $\cc{T}_k^* \cup \bigcup_1^n \{{\la{s}\!}_i\}$ in which each ${\ra{s}\!}_i$ appears as a leaf separation will give a tree-decomposition of the part of the graph at $\sigma$. In particular, since each of the separations in the tree will be nested with $\sigma$, the separators $A_i \cap B_i$ of the separations ${\ra{s}\!}_i$ will lie entirely on one side of every separation in the tree, and so this will in fact be a decomposition of the torso of the part (since any extra edges in the torso lie inside the separators).

Therefore, in practice this tells us that if we have a part in a tree-decomposition whose separators are $\tau_k$-essential then either there is a $k$-tangle in the graph which is contained in that part, or there is a tree-decomposition of the torso of that part with branch-width $<\!k$. In the second case we can then refine the original tree-decomposition by combining it with this new tree-decomposition. By applying this to each inessential part of one of the canonical tree-decompositions formed in \cite{CDHH13CanonicalAlg} we get the following result, which easily implies Theorems \ref{t:maintangle} and \ref{t:mainblock} by taking $\cc{F} = \cc{T}_k^*$ and $\cc{B}_k^*$ respectively.

\begin{corollary}\label{c:tangle}
Let $k \geq 3$ and let $\cc{F} \supseteq \cc{P}_k$ be such that the set $\phi$ of regular $\cc{F}$-tangles is canonical. If $\cc{F}^*$ is defined as in Theorem \ref{t:genprof} then there exists a nested set of separations $\cc{N} \subseteq S_k$ corresponding to an $S_k$-tree $(T,\alpha)$ of $G$ such that:
\begin{itemize}
\item there is a subset $\cc{N}' \subseteq \cc{N}$ that is fixed under every automorphism of $G$ and distinguishes all the regular $\cc{F}$-tangles in $\phi$ efficiently;
\item every vertex $t \in T$ either contains a regular $\cc{F}$-tangle or $\{\alpha(t',t) \, : (t',t) \in \ra{E}T)\} \in \cc{F}^*$.
\end{itemize}
\end{corollary}
\begin{proof}
By Theorem \ref{t:dist} there exists a canonical nested set $\cc{N}'$ of $\phi$-essential separations of order $<\!k$ that distinguishes all the regular $\cc{F}$-tangles in $\phi$ efficiently, and by Theorem \ref{t:genprof} $\phi$ is also the set of $\cc{F}^*$-tangles. Given an inessential part $V_t$ in the corresponding tree-decomposition $(T,\cc{V})$, this part corresponds to some star of separations  $\sigma = \{{\ra{s}\!}_i \, : \, i \in [n]\} \subseteq \cc{N}'$. Each ${\ra{s}\!}_i \in \cc{N}'$ is $\phi$-essential, and, by Theorem \ref{t:genprof}, $\cc{F}^*$ is a standard set of stars which is closed under shifting, and contains $\{\ra{r}\}$ for every co-small $\ra{r}$. Hence, by Lemma \ref{l:main}, if we let $\cc{F}' = \cc{F}^* \cup \bigcup_1^n \{ {\ra{s}\!}_i \} $, there is either an $\cc{F}'$-tangle of $S_k$, or an $S_k$-tree over $\cc{F}'$ in which each ${\ra{s}\!}_i$ appears as a leaf separation.

Suppose that there exists an $\cc{F}'$-tangle $O$. Since $O$ avoids $\cc{F}' \supseteq \cc{F}^*$, it is also an $\cc{F}^*$-tangle, and so $O \in \phi$. By assumption $\cc{N}'$ distinguishes all the regular $\cc{F}$-tangles in $\phi$, so $O$ is contained in some part of the tree-decomposition, and since $O$ avoids $\{\{ {\la{s}\!}_i\} \, : \, i \in [n] \}$, it must extend $\sigma$, and so this part must be $V_t$.  However, this contradicts the assumption that $V_t$ is inessential.

Therefore, by Lemma \ref{l:main}, there exists an $S_k$-tree over $\cc{F}^* \cup \bigcup_1^n \{{\la{s}\!}_i \}$ 
. This gives a nested set of separations $\cc{N}_t$ which contains the set $\sigma$. If we take such a set for each inessential $V_t$ then the set
\[
\cc{N} = \cc{N}' \cup \bigcup_{V_t \text{ inessential}} \cc{N}_t
\]
satisfies the conditions of the corollary.
\end{proof}

We note that, whilst the existence of such a tree-decomposition is interesting in its own right, perhaps a more useful application of Lemma \ref{l:main} is that we can conclude the same for \emph{every} tree-decomposition constructed by the algorithms in \cite{CDHH13CanonicalAlg}. So, we are able to choose whichever algorithm we want to construct our initial tree-decomposition, perhaps in order to have some control over the structure of the essential parts, and we can still conclude that the inessential parts have small branch-width.

Apart from the set $\tau_k$ of $k$-tangles there is another natural set of tangles for which tangle-distinguishing tree-decompositions have been considered. Since a $k$-tangle, as a $\cc{T}_k$-avoiding orientation of $S_k$, induces an orientation on $S_i$ for all $i \leq k$, it induces an $i$-tangle for all $i \leq k$. If an $i$-tangle for some $i$ is not induced by any $k$-tangle with $k > i$ we say it is a \emph{maximal tangle}.

Robertson and Seymour \cite{GMX} showed that there is a decomposition of the graph which distinguishes its maximal tangles, but the theorem does not tell us much about the structure of this tree-decomposition. The approach of Carmesin et al was extended by Diestel, Hundertmark and Lemanczyk \cite{HL2015} to show how an iterative approach to Theorem \ref{t:dist} could be used to build canonical tree-decompositions distinguishing the maximal tangles in a graph (in fact they showed a stronger result for a broader class of profiles which implies the result for tangles). In particular, the results of \cite{HL2015} imply the following.

\begin{theorem}\label{t:maximal}
If $\phi$ is a canonical set of tangles in a graph $G$, then there exists a canonical nested set $\cc{N}(G,\phi)$ of $\phi$-essential separations that distinguishes all the tangles in $\phi$ efficiently.
\end{theorem}

In particular we can apply this to the set of maximal tangles. By looking directly at the proof in \cite{HL2015} one can see the structure of the tree-decomposition formed. The proof proceeds iteratively, by choosing for each $i$ in a turn a nested set of $(i-1)$-separations (that is, separations of order $(i-1)$), which distinguishes efficiently the pairs of $i$-tangles which are distinguished efficiently by an $(i-1)$-separation, such that this set is also nested with the previously constructed sets. 

At each stage in the construction we have a tree-decomposition which distinguishes all the tangles of order $\leq i$ in the graph. Some of these $i$-tangles however will extend to $(i+1)$-tangles in different ways (induced by distinct maximal tangles in the graph). The next stage constructs a nested set of separations distinguishing such tangles, which gives a tree-decomposition of the torsos of the relevant parts. In these tree-decompositions some parts will be `essential', and containing $(i+1)$-tangles, but some will be inessential.

It is natural to expect that the inessential parts constructed at stage $i$ should have branch-width $<i$, by a similar argument as Corollary \ref{c:tangle}. However it is not always the case that the separators of the inessential part satisfy the conditions of Lemma \ref{l:main}, since it can be the case that these inessential parts have separators which are separations constructed in an earlier stage of the process, and as such might not efficiently distinguish a pair of tangles of order $i$.

\begin{question}
Can we bound the branch-width of the inessential parts in such a tree-decomposition in a similar way?
\end{question}
A positive answer to the previous question in the strongest form would give the following analogue of Theorem \ref{t:maintangle}. 
\begin{conjecture}
For every graph $G$ there exists a canonical sequence of tree-decompositions $(T_i,\cc{V}_i)$ for $1 \leq i \leq n$ of $G$ such that
\begin{itemize}
\item $(T_i,\cc{V}_i)$ distinguishes every $i$-tangle in $G$ for each $i$;
\item $(T_n, \cc{V}_n)$ distinguishes the set of maximal tangles in $G$.
\item $(T_{i+1},\cc{V}_{i+1})$ refines $(T_i,\cc{V}_i)$ for each $i$;
\item The torso of every inessential part in $(T_i,\cc{V}_i)$ has branch-width $<i$.
\end{itemize}
\end{conjecture}

\subsection{Proof of Lemma \ref{l:main}}
\begin{proof}[Proof of Lemma \ref{l:main}]
Let us write
\[
\overline{\cc{F}}= \cc{F} \cup \{ \{\la{x}\} \, : \, {\la{s}\!}_i \leq \la{x} \text{ for some } i \in [n] \}.
\]

We first claim that ${\ra{S}\!}_k$ is $\overline{\cc{F}}$-separable. We note that by [\cite{DiestelOumDualityIb}, Lemma 3.4] for every universe $\ra{U}$ and any $k \in \mathbb{N}$, the separation system ${\ra{S}\!}_k$ is separable. Therefore it is sufficient to show that $\overline{\cc{F}}$ is closed under shifting. By assumption $\cc{F}$ is closed under shifting, and the image of any singleton star $\{ \la{x} \} \in \overline{\cc{F}}$ under some relevant $f\downarrow^{\ra{r}}_{\ra{s}}$ is $\{\la{y} \}$ for some separation $\la{x} \leq \la{y}$, and hence $\{ \la{y} \} \in \overline{\cc{F}}$. Therefore, $\overline{\cc{F}}$ is closed under shifting. Furthermore, since $\cc{F}$ was standard, so is $\overline{\cc{F}}$. Hence, we can apply Theorem \ref{t:dual} to $\overline{\cc{F}}$.

By Theorem \ref{t:dual}, either there exists an $S_k$-tree over $\overline{\cc{F}}$, or there exists an $\overline{\cc{F}}$-tangle. Since $\overline{\cc{F}} \supset \cc{F}'$, every $\overline{\cc{F}}$-tangle is also an $\cc{F}'$-tangle, and so in the second case we are done. Therefore we may assume that there exists an $S_k$-tree over $\overline{\cc{F}}$, $(T,\alpha)$. We will use $(T,\alpha)$ to form an $S_k$-tree over $\cc{F}'$.

Since there is no $\overline{\cc{F}}$-tangle, each $\cc{F}$-tangle $O$ must contain some ${\la{s}\!}_i$. We note that, since by assumption $\cc{F}$ contains every co-small separation, $O$ is regular. Hence, since $\sigma$ is a star, this ${\la{s}\!}_i$ is unique. We claim that, for every $\cc{F}$-tangle $O$ such that ${\la{s}\!}_i \in O$ there is some leaf separation $\ra{x} \in \alpha(\ra{E}(T))$ such that $\la{x} \leq {\la{s}\!}_i$. 

Indeed, since $O$ is a consistent orientation of ${\ra{S}\!}_k$, it is contained in some vertex of $(T\alpha)$. However, the star of separations at that vertex, by definition of an $\cc{F}$-tangle, cannot lie in $\cc{F}$, and so must lie in $\overline{\cc{F}} \setminus \cc{F}$. Since each of these stars are singletons, the vertex must be a leaf. Therefore, there is some leaf separation $\ra{x}$ such that $\la{x} \in O$. Since $\{ \la{x} \} \in \overline{\cc{F}} \setminus \cc{F}$, it follows that ${\la{s}\!}_r \leq \la{x}$ for some $r \in [n]$. However, since ${\la{s}\!}_i \in O$, and it was the unique separation in $\sigma$ with that property, it follows that $r=i$, and so ${\la{s}\!}_i \leq \la{x}_i$ as claimed.

If the only leaf separations in $\overline{\cc{F}} \setminus \cc{F}$ were the separations $\{ {\ra{s}\!}_i \, : \, i \in [n]\}$ then $(T,\alpha)$ would be the required $S$-tree over $\cc{F}'$. In general however the tree will have a more arbitrary set $\{  {\ra{x}\!}_{i,j} \}$ of leaf separations (along with some leaf separations arising as separations forced by $\cc{F}$) where ${\la{s}\!}_i \leq {\la{x}\!}_{i,j}$, see Figure \ref{f:tree}. Note that there may not necessarily be any edges in this tree corresponding to the separations $s_i$. 
\begin{figure}[!ht]
\center
\includegraphics[scale=0.15]{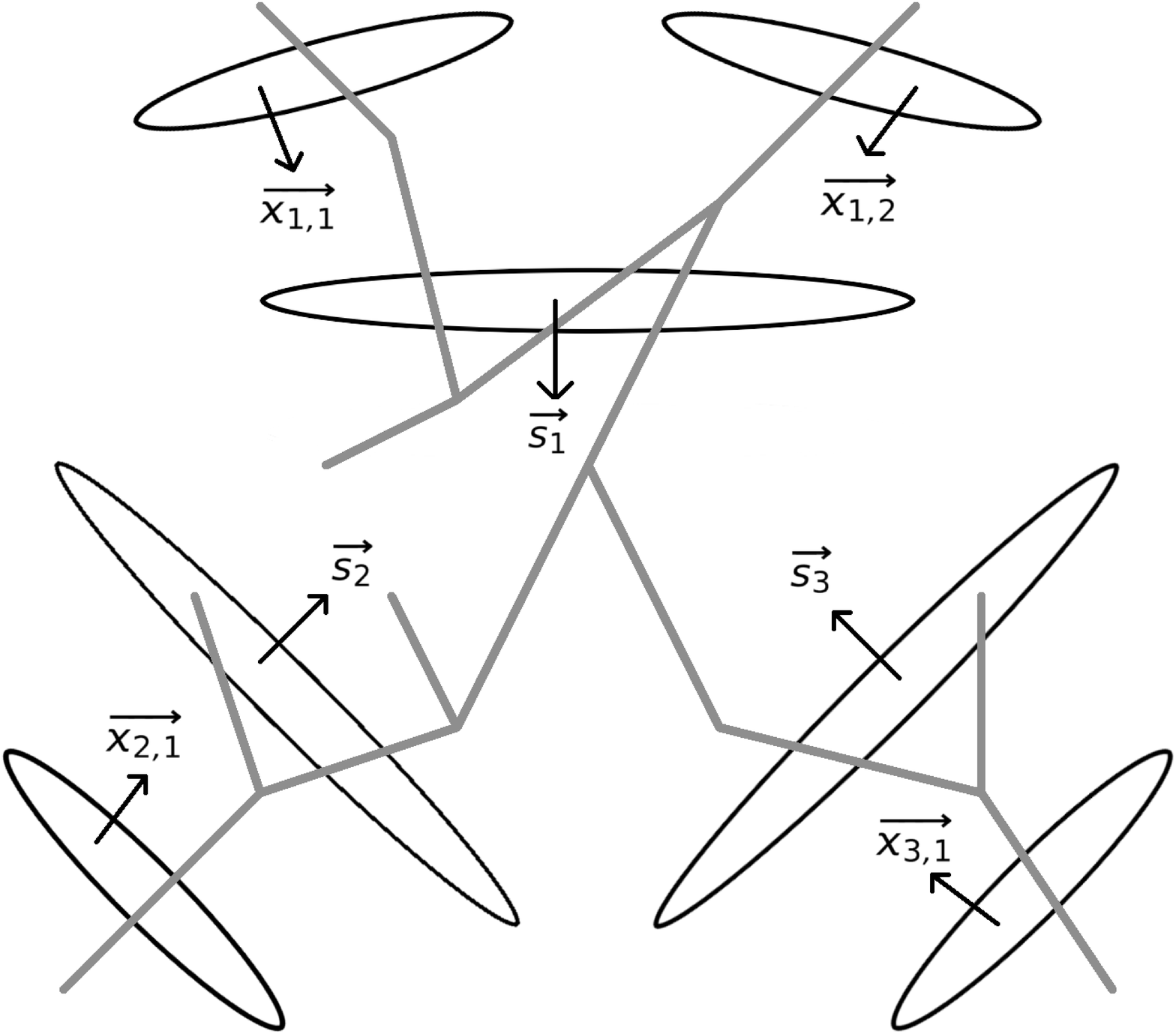}
\caption{The $S_k$-tree over $\overline{\cc{F}}$ with unlabelled leafs corresponding to separations forced by $\cc{F}$.}\label{f:tree}
\end{figure}

We claim that each ${\ra{s}\!}_i$ emulates some ${\ra{x}\!}_{i,j}$ in ${\ra{S}\!}_k$ for $\overline{\cc{F}}$. By assumption, every $s_i \in \sigma$ distinguishes efficiently some pair $O_1$ and $O_2$ of $\cc{F}$-tangles. Suppose that ${\ra{s}\!}_i \in O_1$ and ${\la{s}\!}_i \in O_2$. By our previous claim, there is some leaf separation ${\ra{x}\!}_{i,j}$ such that ${\la{x}\!}_{i,j} \in O_2$. We claim that ${\ra{s}\!}_i$ emulates this ${\ra{x}\!}_{i,j}$ in ${\ra{S}\!}_k$. Note that, since $\overline{\cc{F}}$ is closed under shifting, it would follow that ${\ra{s}\!}_i$ emulates ${\ra{x}\!}_{i,j}$ in ${\ra{S}\!}_k$ for $\overline{\cc{F}}$.  Note that, since ${\ra{s}\!}_i$ and ${\ra{x}\!}_{i,j}$ both distinguish two $\cc{F}$-tangles, they are non-trivial and non-degenerate.

Indeed, given any separation $\ra{r} \geq {\ra{x}\!}_{i,j}$ we have that ${\ra{s}\!}_i \geq {\ra{s}\!}_i \wedge \ra{r} \geq {\ra{x}\!}_{i,j}$  and so ${\ra{s}\!}_i \wedge \ra{r}$ distinguishes $O_1$ and $O_2$. Therefore, since $s_i$ distinguises $O_1$ and $O_2$ efficiently,  $|{\ra{s}\!}_i \wedge \ra{r}| \geq |{\ra{s}\!}_i|$. Hence, by submodularity, $|{\ra{s}\!}_i \vee \ra{r}| \leq |\ra{r}| < k$ and so ${\ra{s}\!}_i \vee \ra{r} \in S_k$. Therefore the image of $f\downarrow^{ {\ra{x}\!}_{i,j}}_{{\ra{s}\!}_i}$ is contained in $S_k$ and so ${\ra{s}\!}_i$ emulates ${\ra{x}\!}_{i,j}$ in ${\ra{S}\!}_k$. Furthermore, since ${\ra{x}\!}_{i,j}$ is non-trivial and non-degenerate, by the comment after Lemma \ref{l:link} we can assume that $T$ is irredundant, and that ${\ra{x}\!}_{i,j}$ is not the image of any other edge in $T$.

Since ${\ra{s}\!}_n$ and ${\ra{x}\!}_{n,j}$ satisfy the conditions of Lemma \ref{l:link}, we conclude that the shift of $(T,\alpha)$ onto ${\ra{s}\!}_n$ is an $S_k$-tree over $\overline{\cc{F}}$ which contains ${\ra{s}\!}_n$ as a leaf separation, and not as the image of any other edge. Let us write $(T_n,\alpha_n)$ for this $S_k$-tree.

If there is some leaf separation $\la{r}$ of $(T_n,\alpha_n)$ such that ${\la{s}\!}_n < \la{r}$ then, since the leaf separations form a star and ${\la{s}\!}_i$ is the image of a unique leaf, we also have that ${\ra{s}\!}_i < \la{r}$. Hence, $\ra{r}$ is trivial, and so $\{\la{r}\}  \in \cc{F}$. Therefore $(T_n,\alpha_n)$ is also an $S_k$-tree over 
\[
\overline{\cc{F}}_n= \cc{F} \cup \{ {\la{s}\!}_n \} \cup \{ \{\la{x}\} \, : \, {\la{s}\!}_i \leq \la{x} \text{ for some } i \in [n-1] \}.
\]

If we repeat this argument for each $1\leq i \leq n$, we end up with a sequence of $S_k$-trees $(T_n,\alpha_n),\newline(T_{n-1},\alpha_{n-1}), \ldots (T_1,\alpha_1)$ over $\overline{\cc{F}}$ such that $(T_j,\alpha_j)$ is also an $S_k$-tree over
\[
\overline{\cc{F}}_j= \cc{F} \cup \{{\la{s}\!}_i \,: \, i \geq j \} \cup \{ \{\la{x}\} \, : \, {\la{s}\!}_i \leq \la{x} \text{ for some } i \in [j-1] \}.
\]
We note that $\overline{\cc{F}}_1 = \cc{F}'$, and so $(T_1,\alpha_1)$ is an $S_k$-tree over $\cc{F}'$, completing the proof.
\end{proof}

\section{Further refining essential parts of tangle-distinguishing tree-decompositions}\label{s:further}
In some sense the tree-decompositions of Corollary \ref{c:tangle} tell us most about the structure of the graph when the essential parts correspond closely to the profiles inside them.  However, as the example in Figure \ref{f:inessential} shows, sometimes there can be essential parts which could be further refined, in order to more precisely exhibit the structure of the graph.

In this section we will discuss how the tools from the paper can be used to achieve this goal. Given a graph $G$ we call a separation $\la{x} \in {\ra{S}\!}_k$ \emph{inessential} if $\la{x} \in O$ for every $k$-tangle $O$ of $G$. Given a $k$-tangle $O$ let $\cc{M}(O)$ be the set of maximal separations in $O$, and let $\cc{M}_I(O)$ be the set of maximal inessential separations. Our main tool will be the following lemma.

\begin{lemma}\label{l:iness}
Let $G$ be a graph, $O$ be a $k$-tangle of $G$ and let $\la{x} \in \cc{M}_I(O)$ be non-trivial. Then there is an $S_k$-tree over $\cc{T}_k^* \cup \{\la{x}\}$.
\end{lemma}
\begin{proof}
As in the proof of Lemma \ref{l:main} let us consider the family of stars
\[
\cc{F} = \cc{T}_k^*\cup \{ \{ \la{r}\} \,:\, \la{x} \leq \la{r} \}.
\]
A similar argument show that this family is standard and closed under shifting, and so Theorem \ref{t:dual} asserts the existence of a $\cc{F}$-tangle, or an $S_k$-tree over $\cc{F}$. As before, an $\cc{F}$-tangle would be a $k$-tangle of $G$ which contains $\ra{x}$, contradicting the fact that $\la{x}$ is inessential. Therefore, there is an $S_k$-tree over $\cc{F}$. However, $O$ must live in some part of this tree-decomposition, and since $O$ is $\cc{T}^*_k$-avoiding it must live in some leaf vertex, corresponding to a singleton star $\{ \la{r} \}$ for some $\la{x} \leq \la{r}$. However, $\la{x}$ was a maximal separation in $O$ and hence $\la{r} \not\in O$ unless $\la{r} = \la{x}$. Therefore the $S_k$-tree is in fact over $\cc{T}_k^* \cup \{\la{x}\}$. 
\end{proof}

Lemma \ref{l:iness} tell us that for every $\la{x} \in \cc{M}_I(O)$ there is a tree-decomposition of the part of the graph behind $\la{x}$ with branch-width $<\!k$. So, we could perhaps hope to refine our canonical $k$-tangle-distinguishing tree-decompositions further using these tree-decompositions. However, there is no guarantee that $\cc{M}_I(O)$ will be nested with the $\tau_k$-essential separations used in a $k$-tangle-distinguishing tree-decomposition, and so we cannot in general refine such tree-decompositions naively in this way. Moreso, in order to decompose as much of the inessential parts of the graph as possible we would like to take such a tree for \emph{each} such maximal separation, however again in general, $\cc{M}_I(O)$ itself may not be nested. 

Our plan will be to find, for each  $\la{x} \in \cc{M}_I(O)$, some inessential separation $\la{u}$ such that $\ra{u}$ emulates $\ra{x}$ in $\ra{S}_k$, that is also nested with the separations from the $k$-tangle-distinguishing tree-decomposition. Furthermore we would like to able to do this in such a way that the separations $\ra{u}$ emulating different maximal separations form a star. Then, for each maximal separation, we could shift the $S_k$-tree given by Lemma \ref{l:iness} to an $S_k$-tree over $\cc{T}^*_k \cup \{ \la{u} \}$. These tree-decompositions could then be used to refine our $k$-tangle-distinguishing tree-decomposition further.  We will in fact show a more general result that may be of interest in its own right.

\subsection{Uncrossing sets of separations}
Given two separations $\ra{r} \leq \ra{s}$ in an arbitrary universe with an order function, we say that $\ra{s}$ is \emph{linked} to $\ra{r}$ if for every $\ra{x} \geq \ra{r}$ we have that
\[
|\ra{x} \vee \ra{s}| \leq |\ra{x}|.
\]
In particular we note that if $\ra{r}, \ra{s} \in{\ra{S}\!}_k$, then $\ra{s}$ being linked to $\ra{r}$ implies that $\ra{s}$ emulates $\ra{r}$ in ${\ra{S}\!}_k$. We first note explicitly a fact used in the proof of Lemma \ref{l:main}.

\begin{lemma}\label{l:min}
Let $(\ra{U},\leq,*,\vee,\wedge)$ be a universe of separations with an order function, and let $\la{s} \leq \la{r}$ be two separations in $\ra{U}$. If $\la{x}$ is a separation of minimal order such that $\la{s} \leq \la{x} \leq \la{r}$, then $\ra{x}$ is linked to $\ra{r}$.
\end{lemma}
\begin{proof}
Given any separation $\ra{y} > \ra{r}$ we note that
\[
\la{s} \leq \la{x} \leq \la{x} \vee \la{y} \leq \la{r},
\]
and so by minimality of $\la{x}$ we have that $|\la{y} \vee \la{x}| \geq |\la{x}|$. Hence, by submodularity $|\ra{y} \vee \ra{x}| = |\la{y} \wedge \la{x}| \leq |\ra{y}|$, and so $\ra{x}$ is linked to $\ra{r}$. Note that, by symmetry, $\la{x}$ is linked to $\la{s}$ also.
\end{proof}

In what follows we will need to use two facts about a universe of separations. The first is true for any universe of separations, that for any two separations $\la{x}$ and $\la{y}$
\[
(\la{x} \wedge \la{y})^* = \ra{x} \vee \ra{y} \text{   and    } (\la{x} \vee \la{y})^* = \ra{x} \wedge \ra{y}.
\]

The second will not be true in general, and so we say a universe of separations is \emph{distributive} if for every three separations $\la{x},\la{y}$ and $\la{z}$ it is true that
\[
(\la{x} \wedge \la{y}) \vee \la{z} = (\la{x} \vee \la{z}) \wedge (\la{y} \vee \la{z}) \text{    and    } 
(\la{x} \vee \la{y}) \wedge \la{z} = (\la{x} \wedge \la{z}) \vee (\la{y} \wedge \la{z}).
\]
It is a simple check that the universe of separations of a graph is distributive.

\begin{lemma}\label{l:pair}
Let $(\ra{U},\leq,*,\vee,\wedge)$ be a distributive universe of separations with an order function, and let ${\la{x}\!}_1$ and ${\la{x}\!}_2$ be two separations in $\ra{U}$. Let ${\la{u}\!}_1$ be any separation of minimal order such that ${\la{x}\!}_1 \wedge {\ra{x}\!}_2 \leq {\la{u}\!}_1 \leq {\la{x}\!}_1$ and let ${\la{u}\!}_2 = {\la{x}\!}_2 \wedge {\ra{u}\!}_1$. Then the following statements hold:
\begin{itemize}
\item ${\ra{u}\!}_1$ is linked to ${\ra{x}\!}_1$ and ${\ra{u}\!}_2$ is linked to ${\ra{x}\!}_2$;
\item $| {\la{u}\!}_1 |\leq | {\la{x}\!}_1 |$ and $| {\la{u}\!}_2 | \leq | {\la{x}\!}_2 |$;
\item ${\la{u}\!}_1 = {\la{x}\!}_1 \wedge {\ra{u}\!}_2$.
\end{itemize}
\end{lemma}
\begin{proof}
We note that ${\ra{u}\!}_1$ is linked to ${\ra{x}\!}_1$ by Lemma \ref{l:min}. We want to show that ${\ra{u}\!}_2$ is linked to ${\ra{x}\!}_2$, that is, given any $\ra{r} > {\ra{x}\!}_2$ we need that $|\ra{r} \vee {\ra{u}\!}_2 |\leq |\ra{r}|$. We first claim that $\ra{r} \vee {\ra{u}\!}_2 = {\la{u}\!}_1 \vee \ra{r}$. Indeed,
\[
\ra{r} \vee {\ra{u}\!}_2 = \ra{r} \vee ({\ra{x}\!}_2 \vee {\la{u}\!}_1)= (\ra{r} \vee {\ra{x}\!}_2) \vee {\la{u}\!}_1 = \ra{r} \vee {\la{u}\!}_1.
\]

We also claim that ${\la{x}\!}_1 \wedge {\ra{x}\!}_2 \leq {\la{u}\!}_1 \wedge \ra{r} \leq {\la{x}\!}_1$. Indeed, ${\la{x}\!}_1 \wedge {\ra{x}\!}_2 \leq {\la{u}\!}_1$ and ${\la{x}\!}_1 \wedge {\ra{x}\!}_2 \leq {\ra{x}\!}_2 \leq \ra{r}$ and so
\[
{\la{x}\!}_1 \wedge {\ra{x}\!}_2 \leq {\la{u}\!}_1 \wedge \ra{r} \leq {\la{u}\!}_1 \leq {\la{x}\!}_1.
\]
Therefore, by minimality of ${\la{u}\!}_1$ we have that $| {\la{u}\!}_1 \wedge \ra{r} | \geq | {\la{u}\!}_1 |$ and so, by submodularity, it follows that
\[
| \ra{r} \vee {\ra{u}\!}_2 | = | {\la{u}\!}_1 \vee \ra{r} | \leq | \ra{r} |,
\]
as claimed.

By minimality of ${\la{u}\!}_1$ we have that $| {\la{u}\!}_1 | \leq | {\la{x}\!}_1 |$. Also we note that, since ${\la{u}\!}_2 = {\la{x}\!}_2 \wedge {\ra{u}\!}_1$ we have that
\[
| {\la{u}\!}_2 | + | {\la{x}\!}_2 \vee {\ra{u}\!}_1 | \leq | {\la{x}\!}_2 | + | {\la{u}\!}_1 |.
\]
However, $| {\la{x}\!}_2 \vee {\ra{u}\!}_1 | = | {\la{u}\!}_1 \wedge {\ra{x}\!}_2 |$, and we claim that
\[
{\la{x}\!}_1 \wedge {\ra{x}\!}_2 \leq {\la{u}\!}_1 \wedge {\ra{x}\!}_2 \leq {\la{x}\!}_1.
\]
Indeed, that second inequality is clear since, ${\la{u}\!}_1 \leq {\la{x}\!}_1$. For the first we note that ${\la{x}\!}_1 \wedge {\ra{x}\!}_2 \leq {\ra{x}\!}_2$, and also ${\la{x}\!}_1 \wedge {\ra{x}\!}_2 \leq {\la{x}\!}_1 \wedge {\ra{u}\!}_2 = {\la{u}\!}_1$, and so ${\la{x}\!}_1 \wedge {\ra{x}\!}_2 \leq {\la{u}\!}_1 \wedge {\ra{x}\!}_2$. Hence, by the minimality of ${\la{u}\!}_1$, we have $| {\la{x}\!}_2 \vee {\ra{u}\!}_1 | \geq | {\la{u}\!}_1 |$. Hence it follows that $ | {\la{u}\!}_2 | \leq | {\la{x}\!}_2 |$, as claimed.

For the last condition, we have that $ {\la{u}\!}_1 \leq {\la{x}\!}_1$ and ${\la{u}\!}_1 \leq {\la{u}\!}_1 \vee {\ra{x}\!}_2 = {\ra{u}\!}_2$, and so ${\la{u}\!}_1 \leq {\la{x}\!}_1 \wedge {\ra{u}\!}_2$. However,
\[
{\la{x}\!}_1 \wedge {\ra{u}\!}_2  = {\la{x}\!}_1 \wedge ( {\ra{x}\!}_2 \vee {\la{u}\!}_1 ) = ( {\la{x}\!}_1 \wedge {\ra{x}\!}_2 ) \vee ( {\la{x}\!}_1 \wedge {\la{u}\!}_1 ) = ( {\la{x}\!}_1 \wedge {\ra{x}\!}_2 ) \vee {\la{u}\!}_1 \leq {\la{u}\!}_1,
\]
and so ${\la{u}\!}_1 = {\la{x}\!}_1 \wedge {\ra{u}\!}_2$.
\end{proof}

We note that if we apply the above lemma to a pair of separations ${\la{x}\!}_1$ and ${\la{x}\!}_2$ such that $x_1$ distinguishes efficiently a pair of regular $k$-profiles, which $x_2$ does not distinguish, say ${\la{x}\!}_1 \in P_1$ and ${\ra{x}\!}_1 \in P_2$ and ${\la{x}\!}_2 \in P_1 \cap P_2$, then ${\la{x}\!}_1$ is of minimal order over all separations ${\la{x}\!}_1 \wedge {\ra{x}\!}_2 \leq {\la{u}\!}_1 \leq {\la{x}\!}_1$. Hence, in Lemma \ref{l:pair}, we can take ${\la{u}\!}_1= {\la{x}\!}_1$ and ${\la{u}\!}_2 = {\la{x}\!}_2 \wedge {\ra{x}\!}_1$. 

Indeed, suppose ${\la{x}\!}_1 \wedge {\ra{x}\!}_2 \leq {\la{u}\!}_1 \leq {\la{x}\!}_1$ is of minimal order. We note that ${\la{u}\!}_1 \in P_1$ by regularity. Similarly, let ${\la{u}\!}_2 = {\la{x}\!}_2 \wedge {\ra{u}\!}_1$ then ${\la{u}\!}_2 \in P_2$ by regularity. Recall that, by Lemma \ref{l:pair} ${\la{u}\!}_1 = {\la{x}\!}_1 \wedge {\ra{u}\!}_2$. Hence, ${\ra{u}\!}_1= {\ra{x}\!}_1 \vee {\la{u}\!}_2 \in P_2$. Therefore, $u_1$ distinguishes $P_1$ and $P_2$ and so, by the efficiency of $x_1$, $| {\la{x}\!}_1 | \leq | {\la{u}\!}_1 |$ as claimed.

The question remains as to what happens for a larger set of separations. It would be tempting to conjecture that the following extension of Lemma \ref{l:pair} holds, where we note that, in general, $(\ra{x} \wedge \ra{y}) \wedge \ra{z} = \ra{x} \wedge (\ra{y} \wedge \ra{z})$ and so, when writing such an expression we can, without confusion, omit the brackets.
\begin{conjecture}
Let $(\ra{U},\leq,*,\vee,\wedge)$ be a distributive universe of separations with an order function, and let $\{{\la{x}\!}_i \, : \, i \in [n]\}$ be a set of separations in $\ra{U}$. Then there exists a set of separations $\{ {\la{u}\!}_i \, : \, i \in [n]\}$ such that the following conditions hold:
\begin{itemize}
\item $\{ {\la{u}\!}_i \, : \, i \in [n]\}$ is a star;
\item ${\ra{u}\!}_i$ is linked to ${\ra{x}\!}_i$ for all $i \in [n]$;
\item $| {\la{u}\!}_i |\leq | {\la{x}\!}_i |$ for all $i \in [n]$;
\item ${\la{u}\!}_i = {\la{x}\!}_i \bigwedge_{j \neq i} {\ra{u}\!}_j$ for all $i \in [n]$.
\end{itemize}
\end{conjecture}

However, it seems difficult to ensure that the fourth condition holds with an inductive argument. We were able to show the following in the case of graph separations, by repeatedly applying Lemma \ref{l:pair}. The extra sets $\{ {\la{r}\!}_i \}$ and $\phi$ appearing in the statement will be useful for the specific application we have in mind, the conclusion when these are empty is the weakened form of the above conjecture.

\begin{lemma}\label{l:star}
Let $G$ be a graph, $k\geq 3$, and let $\phi$ be the set of $k$-profiles in $G$. Suppose that $\{ {\la{r}\!}_i = (A_i,B_i)  \, : \, i \in [n]\} $ is a star composed of $\phi$-essential separations, which distinguish efficiently some set $\phi'$ of regular $k$-profiles and let $\{ {\la{x}\!}_j=(X_j,Y_j)  \, : \, j \in [m]\} \subseteq {\ra{S}\!}_k$ be such that ${\la{x}\!}_j \in P$ for all $j \in [m]$ and $P \in \phi'$. Then there exists a set $\{ {\la{u}\!}_j \, : \, j \in [m]\}$ such that the following conditions hold:

\begin{itemize}
\item $\{ {\la{r}\!}_i \, : \, i \in [n]\} \cup \{ {\la{u}\!}_j \, : \, j \in [m]\} $ is a star;
\item $| {\la{u}\!}_j | \leq | {\la{x}\!}_j |$ for all $j \in [m]$;
\item ${\ra{u}\!}_j$ is linked to ${\ra{x}\!}_j$ for all $j \in [m]$;
\item ${\la{x}\!}_j \bigwedge_i {\ra{r}\!}_i  \bigwedge_{k \neq j} {\ra{x}\!}_k \leq {\la{u}\!}_j \leq {\la{x}\!}_j$ for all $j \in [m]$;
\item $\bigcup_{j=1}^m X_j \cup \bigcup_{i=1}^n A_i = \bigcup_{j=1}^m U_j \cup \bigcup_{i=1}^n A_i $.
\end{itemize}
\end{lemma}
\begin{proof}
Let us start with a set of separations 
\[
Y = \{ {\la{y}\!}_i \,: \, i \in [n+m] \},
\]
and some arbitrary order on the set of pairs $Y^{(2)}$. Initially we set ${\la{y}\!}_i = {\la{x}\!}_j$ for $j \in [m]$ and ${\la{y}\!}_{m+i} = {\la{r}\!}_i$ for $i \in [n]$. For each pair $\{ {\la{y}\!}_i ,{\la{y}\!}_j \}$ in order we apply Lemma \ref{l:pair} to this pair of separations and replace $\{ {\la{y}\!}_i, {\la{y}\!}_j \}$ with the nested pair given by Lemma \ref{l:pair}. After we have done this for each pair, we let ${\la{u}\!}_j := {\la{y}\!}_j$ for each $j \in [m]$.

Note that, since each ${\la{x}\!}_j$ is $\phi$-inessential, and with each application of Lemma \ref{l:pair} we only ever replace a separation by one less than or equal to it, ${\la{y}\!}_j$ is also $\phi$-inessential at each stage of this process for $j \in [m]$. Also, $\{ {\la{r}\!}_i\, : \, i \in [n]\} $ is a star, and so if we apply Lemma \ref{l:pair} to a pair ${\la{r}\!}_i$ and ${\la{r}\!}_k$, neither is changed. Therefore, by the comment after Lemma \ref{l:pair}, we may assume that at every stage in the process ${\la{y}\!}_{m+i} = {\la{r}\!}_i$ for each $i \in [n]$. In particular at the end of the process we have that
\[
Y = \{ {\la{r}\!}_i\, : \, i \in [n]\} \cup \{ {\la{u}\!}_j \, : \, j \in [m]\}.
\]

To see that the first condition is satisfied we note that, given any pair of separations ${\la{y}\!}_i$ and ${\la{y}\!}_j \in Y$, at some stage in the process we applied Lemma \ref{l:pair} to this pair, and immediately after this step we have that ${\la{y}\!}_i \leq {\ra{y}\!}_j$. Since Lemma \ref{l:pair} only ever replaces a separation with one less than or equal to it, it follows that at the end of the process $Y$ is a star. Therefore the family $\{ {\la{r}\!}_i \, : \, i \in [n]\} \cup \{ {\la{u}\!}_j \, : \, j \in [m] \}$ forms a star.

To see that the second condition is satisfied we note that, whenever we apply Lemma \ref{l:pair} we only ever replace a separation with one whose order is less than or equal to the order of the original separation. 

To see that the third condition is satisfied we note that whenever we apply Lemma \ref{l:pair} we only ever replace a separation with one whose inverse is linked to the inverse of the original separation. Therefore it would be sufficient to show that the property of being linked to is transitive. Indeed, suppose that $\ra{r} > \ra{s} > \ra{t}$, $\ra{r}$ is linked to $\ra{s}$ and $\ra{s}$ is linked to $\ra{t}$, all in some separation system $S$. Let $\ra{x} > \ra{t}$ also be in ${\ra{S}\!}_k$. 

However, since $\ra{s}$ is linked to $\ra{t}$, it follows that $|\ra{x} \vee \ra{s}| \leq |\ra{x}|$. Then, since $\ra{x} \vee \ra{s} \geq \ra{s}$ and $\ra{r}$ is linked to $\ra{s}$, it follows that $| (\ra{x} \vee \ra{s}) \vee \ra{r} | \leq |\ra{x} \vee \ra{s}| \leq |\ra{x}|$. However, since $\ra{s} \leq \ra{r}$, $ (\ra{x} \vee \ra{s}) \vee \ra{r} = \ra{x} \vee \ra{r}$, and so $\ra{r}$ is linked to $\ra{t}$.

To see that the fourth condition is satisfied let us consider ${\la{y}\!}_j$ for some $j \in [m]$. There is some sequence of separations ${\la{x}\!}_j = {\la{v}\!}_0 \geq {\la{v}\!}_1 \geq \ldots \geq {\la{v}\!}_t = {\la{u}\!}_j$ that are the values ${\la{y}\!}_j$ takes during this process, corresponding to the $t$ times we applied Lemma \ref{l:pair} to a pair containing the separation ${\la{y}\!}_j$. Suppose that the other separations in those pairs were ${\la{y}\!}_{i_1}, {\la{y}\!}_{i_2}, \ldots , {\la{y}\!}_{i_t}$, and let us denote by ${\la{w}\!}_k$ the value of the separations ${\la{y}\!}_{i_k}$ at the time which we applied Lemma \ref{l:pair} to the pair $\{ {\la{y}\!}_j, {\la{y}\!}_{i_k} \}$.

We claim inductively that for all $0 \leq r \leq t$
\[
{\la{v}\!}_0 \bigwedge_{k=1}^r {\ra{w}\!}_k \leq {\la{v}\!}_r \leq {\la{v}\!}_0.
\]
The statement clearly holds for $r=0$. Suppose it holds for $r-1$. We obtain ${\la{v}\!}_r$ by applying Lemma \ref{l:pair} to the pair $\{ {\la{v}\!}_{r-1}, {\la{w}\!}_r \}$, giving us the pair $\{ {\la{v}\!}_r, {\la{z}\!}\}$. We have that ${\la{v}\!}_{r-1} \wedge {\ra{z}\!} = {\la{v}\!}_r$ and so, since ${\ra{w}\!}_r \leq {\ra{z}\!} $ it follows that
\[
{\la{v}\!}_{r-1} \wedge {\ra{w}\!}_r \leq {\la{v}\!}_r \leq {\la{v}\!}_{r-1}.
\]
 By the induction hypothesis we know that
 \[
{\la{v}\!}_0 \bigwedge_{k=1}^{r-1} {\ra{w}\!}_k \leq {\la{v}\!}_{r-1} \leq {\la{v}\!}_0,
\]
 and so 
\[
{\la{v}\!}_0 \bigwedge_{k=1}^r {\ra{w}\!}_k \leq {\la{v}\!}_{r-1} \wedge {\ra{w}\!}_r  \leq {\la{v}\!}_r \leq {\la{v}\!}_{r-1} \leq {\la{v}\!}_0
\]
as claimed.

For each of the ${\la{w}\!}_k$ there is some separation ${\la{s}\!}_k$ from our original set (that is some ${\la{r}\!}_i$ or ${\la{x}\!}_j$) such that ${\la{w}\!}_k \leq {\la{s}\!}_k$ and so, since ${\ra{s}\!}_k \leq {\ra{w}\!}_k$, and since we apply Lemma \ref{l:pair} to each pair of separations in our original set, we have that
\[
{\la{v}\!}_0 \bigwedge_i {\ra{r}\!}_i  \bigwedge_{k \neq j} {\ra{x}\!}_j \leq {\la{v}\!}_0 \bigwedge_{k=1}^t {\ra{w}\!}_k.
\]

So, recalling that ${\la{v}\!}_0 = {\la{x}\!}_j$ and ${\la{v}\!}_t = {\la{u}\!}_j$, we see that
\[
{\la{x}\!}_j \bigwedge_i {\ra{r}\!}_i  \bigwedge_{k \neq j} {\ra{x}\!}_k \leq {\la{u}\!}_j \leq {\la{x}\!}_j
\]
as claimed.

Finally we note that, if we apply Lemma \ref{l:pair} to a pair of separations $(C,D)$ and $(E,F)$, resulting in the nested pair $\{(C',D'), (E',F')\}$, then 
\[
C \cup E = C' \cup E'.
\]
Indeed, we have that $(C \cap F', D \cup E') = (C',D')$ and $(E \cap D' , F \cup C')=(E',F')$ and so we have that $C' \cup E' = (C \cap F') \cup E' \supseteq C$ and similarly $C' \cup E' = C' \cup (E \cap D') \supseteq E$ and so $C' \cup E' \supseteq C \cup E$. However, since $C' \subseteq C$ and $E' \subseteq E$ we also have $C' \cup E' \subseteq C \cup E$.
\end{proof}

\subsection{Refining the essential parts}
The content of Lemma \ref{l:star} can be thought of as a procedure for turning an arbitrary set of separations into a star which is in some way `close' to the original set, and is linked pairwise to the original set. We note that the second property guarantees us that this star lies in the same $S_k$ as the original set.

Let us say a few words about the other properties of the star which represent this closeness. It will be useful to think about these properties in terms of how we can use this lemma to refine further an essential part in a $k$-tangle-distinguishing tree-decomposition. 

Suppose $\{ {\la{x}\!}_j \, : \, j \in [m]\} = \cc{M}_I(O)$ for some $k$-tangle $O$, and $\{ {\la{r}\!}_i \, : \, i \in [n]\}$ is the star of separations at the vertex where $O$ is contained in a tree-decomposition, specifically one where each $r_i$ distinguishes efficiently some pair of $k$-tangles. By applying Lemma \ref{l:star} we get a star $\{ {\la{u}\!}_j \, : \, j \in [m]\}$ satisfying the conclusions of the lemma. For each non-trivial ${\la{x}\!}_j$, by Lemma \ref{l:iness}, there exists an irredundant $S_k$-tree over $\cc{T}_k^* \cup \{ {\la{x}\!}_j \}$ containing ${\la{x}\!}_j$ as a leaf separation, such that ${\la{x}\!}_j$ is not the image of any other edge. We can then use Lemma \ref{l:link} to shift each of these $S_k$-trees onto ${\ra{u}\!}_j$, giving us an $S_k$-tree over $\cc{T}_k^* \cup \{ {\la{u}\!}_j \}$ . If ${\la{x}\!}_j$ is trivial then so is ${\la{u}\!}_j$, and so there is an obvious $S_k$-tree over $\cc{T}_k^* \cup \{ {\la{u}\!}_j \}$ containing ${\ra{u}\!}_j$ as a leaf separation, that with a single edge corresponding to $u_j$.

Doing the same for each $k$-tangle in the graph and taking the union all of these $S_k$-trees, together with the tree-decomposition from Corollary \ref{c:tangle}, will give us a refinement of this tree-decomposition which maintains the property of each inessential part being too small to contain a $k$-tangle, but also further refines the essential parts. The properties of the star given by Lemma \ref{l:star} give us some measurement of how effective this process is in refining the essential parts of the graph. 

We first note that, given a $k$-tangle $O$, which is contained in some part $V_t$ of a $k$-tangle-distinguishing tree-decomposition, by the fifth property in Lemma \ref{l:star} every vertex in the part $V_t$ which lies on the small side of some maximal inessential separation in $O$ will be in some inessential part of this refinement. 

However this property is also satisfied by the rather naive refinement formed by just taking the union of some small separations $(A_i,V)$ with the $A_i$ covering the same vertex set. The problem with this naive decomposition is it does not really refine the part $V_t$, since there is a still a part with vertex set $V_t$ in the new decomposition. Ideally we would like our refinement to make this essential part as small as possible, to more precisely exhibit how the $k$-tangle $O$ lies in the graph.

Our refinement comes some way towards this, as evidenced by the fourth condition . For example if we have some separation $\la{s} = (A,B)$ which lies `behind' some maximal inessential separation in $O$, that is $\la{s} \leq {\la{x}\!}_j$ for some $j$, and is nested $\cc{M}_I(O) \cup \{ {\la{r}\!}_i \, : \, i \in [n]\}$, then it is easy to check that the fourth property guarantees it will also lie behind some ${\la{u}\!}_k$ given by Lemma \ref{l:star}. So, in the refined tree-decomposition, the part containing $O$ will not contain any vertices that lie strictly in the small side of such a separation, $A \setminus B$.

Suppose $\{ {\la{r}\!}_i \, : \, i \in [n]\}$ is an essential part in a tree-decomposition $(T,\cc{V})$ containing a tangle $O$. We say a vertex $v \in V$ is \emph{inessentially separated from $O$ relative to $(T,\cc{V})$} if there is a separation $(A,B)$ which is nested with $\cc{M}_I(O) \cup \{ {\la{r}\!}_i\, : \, i \in [n]\}$ such that $v \in A \setminus B$, and there exists some $(X,Y) \in \cc{M}_I(O)$ such that $(A,B) \leq (X,Y)$. For example, in Figure \ref{f:inessential}, the vertices in the long paths are inessentially separated from the tangles corresponding to the complete subgraphs relative to the canonical tangle-distinguishing tree-decomposition.

\begin{theorem}\label{t:min}
For every graph $G$ and $k \geq 3$ there exists a tree-decomposition $(T,\cc{V})$ of $G$ of adhesion $<\!k$ with the following properties
\begin{itemize}
\item The tree-decomposition $(T',\cc{V}')$ induced by the essential separations is canonical and distinguishes every $k$-tangle in $G$;
\item The torso of every inessential part has branch-width $<\!k$.
\item For every essential part $V_t$ which contains a tangle $O$, there are no vertices $v \in V_t$ which are inessentially separated from $O$ relative to $(T',\cc{V}')$.
\end{itemize}
\end{theorem}

Given a vertex $v \in V$ we say that that $x$ is \emph{well separated} from $O$ if there is a separation $(A,B)$ which is nested with $\cc{M}(O)$ such that $v \in A \setminus B$, and there exists some $(X,Y) \in \cc{M}(O)$ such that $(A,B) \leq (X,Y)$. 

We can think of the vertices which are well separated from $O$ as being `far away' from $O$ in the graph. Indeed, if $\cc{M}(O)$ is a star, then $O$ is a $k$-block, and the set of vertices well separated from $O$ are just the vertices not in the $k$-block. In general a tangle will not correspond as closely to a concrete set of vertices as a $k$-block, and crossing separations in $\cc{M}(O)$ somehow demonstrate the uncertainty of whether a vertex `lives in' $O$ or not. However, if a separation $(A,B) \in O$ is nested with $\cc{M}(O)$, then $O$ should be in some way fully contained in $B$, and so the vertices in $A \setminus B$ are `far away' from $O$. 

\begin{question}
For every graph $G$, does there exist a tree-decomposition which distinguishes the $k$-tangles in a graph, whose essential parts are small in the sense that for each $k$-tangle $O$, there is no vertex $x$ which can be well separated from $O$ in the part of the tree-decomposition which contains $O$? Does there exist such a tree-decomposition with the further property that the inessential parts have branch-width $<\!k$?
\end{question}

\bibliographystyle{plain}
\bibliography{refine}
\end{document}